\documentclass[12pt, a4paper, reqno]{amsart}
\usepackage{amscd, amssymb}
\usepackage{amsthm}
\usepackage{tikz}
\usetikzlibrary{arrows,positioning} 
\usepackage{mathrsfs}
\usepackage{bbm}

\usepackage{multirow}
\usepackage{makecell}

\usepackage{hyperref}
\hypersetup{
	colorlinks=true,
	linkcolor=blue,
	filecolor=magenta,  
	urlcolor=cyan,
}

\usepackage{hyperref}
\usepackage[capitalise]{cleveref}

\def\dim{\operatorname{dim}}

\def\Ind{\operatorname{Ind}}

\def\C{\mathbb{C}}
\def\Q{\mathbb{Q}}
\def\R{\mathbb{R}}
\def\N{\mathbb{N}}
\def\Z{\mathbb{Z}}

\def\KK{\mathcal{K}}

\def\UU{\mathcal{U}}

\def\a{\mathfrak{a}}
\def\b{\mathfrak{b}}
\def\c{\mathfrak{c}}
\def\m{\mathfrak{m}}
\def\n{\mathfrak{n}}
\def\p{\mathfrak{p}}
\def\q{\mathfrak{q}}
\def\g{\mathfrak{g}}
\def\h{\mathfrak{h}}

\def\k{\mathfrak{k}}
\def\l{\mathfrak{l}}
\def\s{\mathfrak{s}}
\def\o{\mathfrak{o}}

\def\ol{\overline}

\def\sub{\subseteq}

\newtheorem{thm}{Theorem}[section]
\newtheorem{cor}[thm]{Corollary}
\newtheorem{lemma}[thm]{Lemma}
\newtheorem{prop}[thm]{Proposition}
\newtheorem{conj}[thm]{Conjecture}

\theoremstyle{definition}
\newtheorem{definition}[thm]{Definition}
\theoremstyle{remark}
\newtheorem{remark}[thm]{Remark}
\newtheorem{example}[thm]{Example}

\numberwithin{equation}{section}

\usepackage{relsize}
\usepackage{fancyhdr}
\usepackage[all,cmtip]{xy}

\begin{document}
	
	\title{On the Cartan-Helgason theorem for supersymmetric pairs}
	
	\author[Alexander Sherman]{Alexander Sherman}
	
	\begin{abstract}
		Let $(\g,\k)$ be a supersymmetric pair arising from a finite-dimensional, symmetrizable Kac-Moody superalgebra $\g$.  An important branching problem is to determine the finite-dimensional highest weight $\g$-modules which admit a $\k$-coinvariant, and thus appear as functions in a corresponding supersymmetric space $\mathcal{G}/\KK$.  This is the super-analogue of the Cartan-Helgason theorem.  We solve this problem by generalizing the theory of odd reflections for superalgebras to that of singular reflections, and study their effect on representations.  This allows us to reduce to the case of a rank one pair.  An explicit presentation of spherical weights is provided for every pair when $\g$ is indecomposable.
	\end{abstract}
	
	\maketitle
	\pagestyle{plain}

	\section{Introduction}	
	  Let $\g$ be a simple complex Lie algebra, and let $\k\sub\g$ be a symmetric subalgebra, i.e.~$\k$ is the fixed points of an involution on $\g$.  The Cartan-Helgason theorem (see \cite{H1}) describes which irreducible $\g$-modules admit a $\k$-invariant (co)vector, and thus define functions on a corresponding symmetric space $G/K$.  As a generalization of the Peter-Weyl theorem, this allows one to describe the space of polynomial functions on $G/K$, and similarly to describe the $L^2$ functions on $G_c/K_c$, where $G_c$ and $K_c$ are compact real forms of $G$ and $K$, respectively.
	  
	  \subsection{Questions in the super-setting}  We are interested in the analogous question in the super setting.  Let $\g$ be a symmetrizable Kac-Moody superalgebra, i.e.~$\g=\g\l(m|n),\o\s\p(m|2n)$, or simple exceptional.  (We work with $\g\l(m|n)$ instead of $\s\l(m|n)$ for simplicity.)  Let $\theta$ be an involution on $\g$  preserving an invariant form, and giving fixed points $\k$ and $(-1)$-eigenspace $\p$.  We call $(\g,\k)$ a supersymmetric pair. Broadly speaking, we are interested in the following branching problem:
	  
	  \
	  
	  \textbf{Question A:} For which finite-dimensional, indecomposable $\g$-modules $V$ do we have \linebreak $(V^*)^{\k}\neq0$?
	  
	  \
	  
	  By Frobenius reciprocity, this question is directly related to when we can realize $V$ inside $\C[\mathcal{G}/\KK]$ (the algebra of polynomial functions on $\mathcal{G}/\KK$), where $\mathcal{G},\KK$ are global forms of $\g$ and $\k$.  We immediately see a difference between the super setting and the classical setting in that representations need not be semisimple, forcing us to consider indecomposable representations.
	  
	  A full answer to question A is extremely challenging, if not hopeless, if only because of the wild nature of representations of $\g$.  However, in this text we make progress in answering question A in the case when $V$ is simple, and more generally when it is a highest weight module, meaning a module with a cyclic highest weight vector.  We write these questions more explicitly for later reference.
	  
	  \
	  
	  \textbf{Question B:} For which finite-dimensional highest weight $\g$-modules $V$ do we have \linebreak $(V^*)^{\k}\neq0$?
	  
	  \
	  
	  \textbf{Question C:} For which finite-dimensional simple $\g$-modules $V$ do we have $(V^*)^{\k}\neq0$?
	  
	  \
	  
	  Questions B and C are of considerable importance as branching problems in representation theory, and a full understanding of their answer would yield great insight into Question A and the structure of $\C[\mathcal{G}/\KK]$.
	  
	  As already stated, the answer to these questions in the classical situation is given by the Cartan-Helgason theorem, see \cite{H1}.  In that case, representations are completely reducible, and thus one only needs to look at simple modules, making questions A, B, and C equivalent.
	  
	  \subsection{Previously known results} Let $\a\sub\p_{\ol{0}}$ be a maximal abelian subspace (a Cartan subspace).  Then we obtain a restricted root system $\Delta\sub\a^*$, and the choices of positive systems for $\Delta$ are equivalent to choices of simple roots $\Sigma\sub\Delta$, which we also refer to as a base. 
	  
	  Given a base $\Sigma\sub\Delta$, let $\n_{\Sigma}$ be the corresponding nilpotent subalgebra, which is generated by $\g_{\alpha}$ for $\alpha\in\Sigma$.  Then we \textbf{assume} the Iwasawa decomposition holds, i.e.~that $\g=\k\oplus\a\oplus\n_{\Sigma}$.  Let $\b$ be any Borel subalgebra of $\g$ containing $\a\oplus\n$; we call $\b$ an Iwasawa Borel subalgebra of $\g$.  Let $\q_{\Sigma}=\c(\a)\oplus\n_{\Sigma}$, a parabolic subalgebra which contains any Iwasawa Borel subalgebra that contains $\a\oplus\n_{\Sigma}$.  (Here $\c(\a)$ denotes the centralizer of $\a$ in $\g$.)
	  
	  In \cite{AS} it was shown that if $V$ is a $\b$-highest weight module of weight $\lambda$ with respect to an Iwasawa Borel $\b$ such that $(V^*)^{\k}\neq0$, then in fact $\lambda\in\a^*$, and the $\b$-highest weight vector is actually a $\q_{\Sigma}$-eigenvector.  Further, $\dim(V^*)^{\k}\leq1$. The proof of this result works just like in the classical setting of the Cartan-Helgason theorem.  Since the parabolic $\q_{\Sigma}$ is determined by our base $\Sigma$, we define $P_{\Sigma}^+\sub\a^*$ to be the \emph{$\Sigma$-spherical weights}, i.e.~those $\lambda\in\a^*$ for which there exists a finite-dimensional $\g$-module $V$ of highest weight $\lambda$ with respect to $\q_{\Sigma}$ such that $(V^*)^{\k}\neq0$. 
	  
	  In \cite{AS} they use a super-generalization of the Harish-Chandra $c$-function to prove a partial converse: if $\lambda\in\a^*$ is integral and `high enough' with respect to $\Sigma$, in a sense we do not discuss here, then in fact $\lambda\in P_{\Sigma}^+$.
	  
	  \subsection{Main results}  We improve on the work in \cite{AS} by computing entirely the $\Sigma$-spherical weights $P_{\Sigma}^+$ with respect to particular bases $\Sigma$ for every supersymmetric pair.  We state this more precisely below:
	  \begin{thm}
	  	For the following supersymmetic pairs, we determine $P_{\Sigma}^+$ for every base $\Sigma\sub\Delta$, thus giving a full answer to question B:
	  	\[
	  	(\g\l(m|n),\g\l(r|s)\times\g\l(m-r|n-s)), \ \ \ (\o\s\p(m|2n),\o\s\p(r|2n)\times\o\s\p(m-r|2n)),
	  	\]
	  	\[
	  	(\o\s\p(m|2n),\o\s\p(m|2s)\times\o\s\p(m|2n-2s)), \ \ \	(\mathfrak{d}(2,1:a),\o\s\p(2|2)\times\s\o(2)),
	  	\]
	  	\[
       (\a\b(1|3),\g\o\s\p(2|4)) \ \ \ (\a\g(1|2),\mathfrak{d}(2,1;a)), \ \ \ (\a\b(1|3),\mathfrak{d}(2,1;2)\times\s\l(2)).
	  	\]
	  	For the remaining supersymmetric pairs, which are
	  	\[
	  	 \ \ \ (\o\s\p(2m|2n),\g\l(m|n)), \ \ \	(\g\l(m|2n),\o\s\p(m|2n)), 
	  	\]
	  	\[
	  	(\o\s\p(m|2n),\o\s\p(m-r|2n-2s)\times\o\s\p(r|2s)), \ \ \ (\a\b(1|3),\s\l(1|4)),
	  	\]
	  	we compute the spherical weights with respect to certain positive systems.  Thus we obtain the answer to question B for certain positive systems.
	  \end{thm}  
 
 	See Section \ref{section tables} for tables describing the sets $P_{\Sigma}^+$ explicitly for each pair $(\g,\k)$.
 	
 	\subsection{Method of computation} Our method of computation is directly inspired by the ideas of Serganova used in \cite{S} to compute the dominant weights of a Kac-Moody superalgebra $\g$.  Indeed, her result is a special case of the arguments we give in the case of the diagonal pair $(\g\times\g,\g)$.  
 	
 	The idea is that for any base $\Sigma\sub\Delta$ and any $\lambda\in\a^*$, we may define a $\g$-module $V_{\Sigma}(\lambda)$ which is of highest weight $\lambda$ with respect to $\Sigma$ and admits a $\k$-coinvariant (i.e.~$(V_{\Sigma}(\lambda)^*)^{\k}\neq0$).  Further, we have that $\lambda\in P_{\Sigma}^+$ if and only if $V_{\Sigma}(\lambda)$ is finite-dimensional, i.e.~integrable.  Checking integrability must be done on all so-called `principal roots' $\Pi\sub\Delta$, which are those that generate the even part $\Delta_0$ of the root system $\Delta$.  If $\Pi\sub\Sigma$, which happens in certain cases (e.g.~for the `standard' choice of $\Sigma$ when $(\g,\k)=(\g\l(m|2n),\o\s\p(m|2n))$), then integrability becomes easy to check, giving a straightforward description of $P_{\Sigma}^+$. 
 	 	
 	\subsection{Singular reflections} In most cases we have $\Pi\not\sub\Sigma$, and thus we must reflect $\Sigma$ in so-called \emph{singular roots} to deal with principal roots not lying in $\Sigma$.  We say $\alpha\in\Delta$ is a singular root if both $\alpha$ and $2\alpha$ are not restrictions of even roots of $\g$. Singular roots come in 2 flavours: isotropic, meaning that $(\alpha,\alpha)=0$, and non-isotropic, so that $(\alpha,\alpha)\neq0$.   If $\alpha\in\Sigma$ is singular, we define $r_{\alpha}\Sigma$ to be the base associated to the positive system $(\Delta\setminus\{\alpha\})\sqcup\{-\alpha\}$.  See Lemma \ref{lemma sing reflection on base} for an explicit description of $r_{\alpha}\Sigma$.
 	
 	As we show, if $\alpha$ is singular isotropic, then the reflection in $\alpha$, which we write as $r_{\alpha}$, is well behaved on representations.  Namely, we have $V_{\Sigma}(\lambda)\cong V_{r_{\alpha}\Sigma}(r_{\alpha}\lambda)$, where $r_{\alpha}\lambda=\lambda$ if $(\lambda,\alpha)=0$, and $r_{\alpha}\lambda=\lambda-2\alpha$ if $(\lambda,\alpha)\neq0$. 
 		
 	Things become more interesting when $\alpha$ is singular non-isotropic.  In this case $\g_{\alpha}$ will be of dimension $(0|2n)$ for some $n\in\Z_{\geq0}$.  Write $h_{\alpha}\in\a^*$ for the coroot of $\alpha$.  Then if $\lambda(h_{\alpha})/2\notin\{n+1,\dots,2n\}$, we have that $V_{\Sigma}(\lambda)\cong V_{r_{\alpha}\Sigma}(r_{\alpha}\lambda)$, where the formula for $r_{\alpha}\lambda$ is given in Lemma \ref{lemma noniso noncrit reflection}.
 	
 	However if $\lambda(h_{\alpha})/2=n+k\in\{n+1,\dots,2n\}$, then $V_{\Sigma}(\lambda)$ must contain $V_{\Sigma}(\lambda-2k\alpha)$, implying that it is never simple, and further it is not highest weight with respect to $r_{\alpha}\Sigma$.  In this case we say that $\lambda$ is an $\alpha$-\emph{critical weight}. Reflecting $\alpha$-critical weights to other simple roots systems becomes a tricky business.  Nevertheless, if one performs only one simple reflection in a singular root, we do have control over what happens: see Lemma \ref{lemma refl noniso crit}.  This allows us to compute $P_{\Sigma}^+$ in cases where non-isotropic singular roots appear.
 
	\subsection{Consequences for simple spherical modules} As explained in the previous paragraph, our work has an ad-hoc element to it when reflecting $\alpha$-critical weights. This prevents us from computing spherical weights with respect to arbitrary positive systems in every case.  However, if we are interested in Question C, such issues don't arise, because if $\lambda$ is $\alpha$-critical then necessarily $V_{\Sigma}(\lambda)$ is not simple.  Hence, if $V_{\Sigma}(\lambda)$ is finite-dimensional and simple, we must have $V_{\Sigma}(\lambda)\cong V_{r_{\alpha}\Sigma}(r_{\alpha}\lambda)$ for any $\alpha\in\Sigma$.  
	
	If $\Sigma$ is a base, we call $\lambda\in P_{\Sigma}^+$ \emph{fully reflectable} if for any base $\Sigma'\sub\Delta$, we have $V_{\Sigma}(\lambda)\cong V_{\Sigma'}(\lambda_{\Sigma'})$ for some $\lambda_{\Sigma'}\in P_{\Sigma'}^+$.  
	
	\begin{conj}\label{conj 1}
		Let $\Sigma\sub\Delta$ be a base, and let $\lambda\in P_{\Sigma}^+$ be fully reflectable.  Then $V_{\Sigma}(\lambda)$ is simple if and only if for any base $\Sigma'\sub\Delta$ and any non-isotropic $\beta\in\Sigma'$ with \linebreak $m(\beta):=-\operatorname{sdim}(\g_{\beta})/2-\operatorname{sdim}\g_{2\beta}\in\Z_{\geq0}$, we have
		\[
		\lambda_{\Sigma'}(h_{\beta})/2\notin\{m(\beta)+1,\dots,2m(\beta)\}.
		\]
	\end{conj}
	
	The numerical conditions in the above conjecture arise from considering the rank one cases, and studying when two dominant weights will have the same eigenvalue for the Casimir.  We do not prove the necessity of these conditions in the article.  The proof for the case of $(\o\s\p(m|2n),\o\s\p(m-1|2n))$ is written in \cite{Sh2}.  We note that Conjecture \ref{conj 1} was shown to hold under an extra genericity hypothesis on $\lambda$ in Sec.~6.4 of \cite{Sh3}.
	
%
%
	\subsection{Supersymmetric spaces}  Our work is part of an ongoing project to improve understanding of supersymmetric spaces and their connections to representation theory.  Supersymmetric spaces are homogeneous superspaces of the form $\mathcal{G}/\KK$, where $\KK$ is a symmetric subgroup of the supergroup $\mathcal{G}$.  Such spaces are natural in the study of super harmonic analysis, see for instance \cite{C} and \cite{LS}.  They also have important connections to interpolation polynomials (see \cite{SSS}, \cite{SSS2}, and \cite{SaSo}), integrable systems (see \cite{SeV}), and physics (see \cite{Z} and \cite{SRFL}).  	
	
	\subsection{Outlook for the queer Kac-Moody setting}  We expect our techniques to generalize to the queer Kac-Moody setting (see \cite{ShSi}), which includes the supersymmetric pairs $(\q(n),\q(r)\times\q(n-r))$.  This will be the subject of future work.
	
	\subsection{Outline}  In Section 2 we develop the necessary facts about restricted root systems we will use, including about singular reflections.  Section 3 studies the modules $V_{\Sigma}(\lambda)$, how they behave under singular reflections, and the tools for checking integrability.  Section 4 explicitly describes the sets $P_{\Sigma}^+$ for each supersymmetric pair and a choice of base $\Sigma$.
	
	\subsection{Acknowledgements}  We thank anonymous referees for their helpful comments and suggestions. Further, we thank Vera Serganova for providing feedback on an initial version of this article, and for many stimulating discussions.  We also thank Shifra Reif, Siddhartha Sahi, and Hadi Salmasian for many helpful discussions about supersymmetric pairs and restricted root systems. This project was partially made possible by a SQuaRE at the American Institute for Mathematics, and we thank AIM for providing a supportive and mathematically rich environment. The author was partially supported by ARC grant DP210100251 and by a Simons-AMS Travel Grant.

	\section{Restricted root systems} 
	
	 In what follows, for a super vector space $V$ we write $V=V_{\ol{0}}\oplus V_{\ol{1}}$ for its parity decomposition.  We always work over an algebraically closed field $\Bbbk$ of characteristic 0. 
	
	\subsection{Supersymmetric pairs}
	Let $\g$ be a finite-dimensional, symmetrizable Kac-Moody Lie superalgebra (see \cite{S}). Let $\theta$ be an involution of $\g$ which preserves a nondegenerate, invariant bilinear form $(-,-)$ on $\g$.  Write the eigenspace decomposition for $\theta$ as $\g=\k\oplus\p$, where $\k=\g^\theta$.  Then $\k$ is naturally a subalgebra, and $\p$ is a $\k$-module. We call $(\g,\k)$ a supersymmetric pair.
	
	Set $\a\sub\p_{\ol{0}}$ to be a Cartan subspace, meaning a maximal abelian subspace. It is known that $\a$ is unique to conjugacy by the action of $\exp(\k_{\ol{0}})\sub GL(\p_{\ol{0}})$ (see Sec.~26 of \cite{T}).  Set $\m=\k\cap\c(\a)$ to be the centralizer of $\a$ in $\k$.
	
	\
	
	\textbf{Assumption ($\star$):} We assume that $\c(\a)=\a\oplus\m$.  Equivalently, $(\g,\k)$ admits an Iwasawa decomposition (see next paragraph).
	
	\
	
	By the classical picture (see Sec.~26 of \cite{T}), we have $\c(\a)_{\ol{0}}=\a\oplus\m_{\ol{0}}$.  Thus our assumption is equivalent to $\c(\a)_{\ol{1}}=\m_{\ol{1}}$, i.e.~$\c(\a)_{\ol{1}}\sub\k$.  By \cite{Sh}, if $\g$ is indecomposable then either $\c(\a)_{\ol{1}}\sub\k$ or $\c(\a)_{\ol{1}}\sub\p$, and $(\g,\k)$ admits an Iwasawa decomposition (see Lemma \ref{lemma iwasawa}) if and only if $\c(\a)_{\ol{1}}\sub\k$.
	
	If $\h$ is a Cartan subalgebra of $\g$ containing $\a$, then by Cor.~26.13 of \cite{T}, $\h$ is $\theta$-stable.  We fix a choice of such a Cartan subalgebra $\h$ throughout.
	
	\subsection{Restricted root system} We may consider the action of $\a$ on $\g$ by the adjoint action.  Writing $\Delta\sub\a^*\setminus\{0\}$ for the non-zero weights of this action, we have
	\[
	\g=\c(\a)\oplus\bigoplus\limits_{\alpha\in\Delta}\g_{\alpha}=\m\oplus\a\oplus\bigoplus\limits_{\alpha\in\Delta}\g_{\alpha},
	\]
	where we have used out assumption ($\star$) for the last equality.  We refer to elements of $\Delta$ as restricted roots, or just roots when the context is clear.
	
	\begin{remark}
		Let $\widetilde{\Delta}\sub\h^*$ denote the root system of $\g$.  Then another description of $\Delta$ may be given as
		\[
		\Delta=\{\alpha|_{\a}:\alpha\in\widetilde{\Delta}\}\setminus\{0\}.
		\]
	\end{remark}

	The following definition is from Sec.~5 of \cite{Sh}; we write $\Z\Delta$ for the abelian subgroup of $\a^*$ generated by $\Delta$.
	\begin{definition}
		Let $\phi:\Z\Delta\to\R$ be a group homomorphism such that $\phi(\alpha)\neq0$ for all $\alpha\in\Delta$.  Then we write $\Delta^+:=\{\alpha\in\Delta|\phi(\alpha)>0\}$, and call $\Delta^+\sub\Delta$ a choice of positive system of $\Delta\sub\a^*$.
	\end{definition}

	\begin{definition}
		A base $\Sigma\sub\Delta$ is a linearly independent set in $\a^*$ such that 
		\[
		\Delta\sub\N\Sigma\sqcup(-\N\Sigma).
		\]
		In other words, every $\alpha\in\Delta$ is either a non-negative or non-positive integral linear combination of elements of $\Sigma$.  We call elements of a base \emph{simple (restricted) roots}.  
	\end{definition}

	\begin{definition}
		The rank of a supersymmetric pair $(\g,\k)$ is the size of any base $\Sigma\sub\Delta$ of the restricted root system.
	\end{definition}

	Given a base $\Sigma\sub\Delta$, set $\Delta^+_{\Sigma}:=\N\Sigma$. 
	
	\begin{lemma}
			The set $\Delta_{\Sigma}^+$ is a positive system.  Further, the correspondence $\Sigma\mapsto\Delta_{\Sigma}^+$ is bijective.
	\end{lemma}  
	\begin{proof}
		For $\Sigma\sub\Delta$ a base, define $\phi_{\Sigma}:\Z\Delta\to\R$ by $\phi(\alpha)=1$ for $\alpha\in\Sigma$, and extend linearly.  Then it is clear that for $\phi_{\Sigma}$ the corresponding positive system is $\Delta^+_{\Sigma}$.  The injectivity of the correspondence $\Sigma\mapsto\Delta^+_{\Sigma}$ follows from the definition of a base.
		
		For surjectivity, Sec.~5 of \cite{Sh} explains that we may extend any positive system $\Delta^+\sub\Delta$ to a choice of positive system for all $\g$ with simple roots $\widetilde{\Sigma}\sub\h^*$.  Let us call $\Sigma$ the projection of $\widetilde{\Sigma}$ to $\a^*$. Then $\Sigma$ will be a base by Prop.~5.7 and Lem.~5.10 of \cite{Sh}, and it is clear that $\Delta^+=\Delta_{\Sigma}^+$.
	\end{proof}
	
	\subsection{Nilpotent subalgebra $\n_{\Sigma}$ and Iwasawa decomposition} Given a base $\Sigma$, we obtain a nilpotent subalgebra $\n_{\Sigma}$, which by definition is 
	\[
	\n_{\Sigma}^+=\bigoplus\limits_{\alpha\in\Delta_{\Sigma}^+}\g_{\alpha}.
	\]
	
	\begin{lemma}\label{lemma iwasawa}
		For any base $\Sigma$ we have the Iwasawa decomposition $\g=\k\oplus\a\oplus\n_{\Sigma}^+$.
	\end{lemma}
	\begin{proof}
		This follows from our assumption that $\c(\a)_{\ol{1}}\sub\k$ and a standard argument (see Thm.~5.3 of \cite{Sh}).
	\end{proof}

	\begin{lemma}\label{lemma n gend by}
	 	The subalgebra $\n_{\Sigma}^+$ is generated by the subspaces $\g_{\alpha}$ for $\alpha\in\Sigma$.
	\end{lemma}
	\begin{proof}
		Extending $\a$ to a Cartan subalgebra $\h$ of $\g$, by Sec.~5 of \cite{Sh} we may lift $\Sigma$ to a base $\widetilde{\Sigma}\sub\h^*$ such that every element $\alpha\in\widetilde{\Sigma}$ is either fixed by $\theta$, and thus $\g_{\alpha}\sub\m$, or $\theta\alpha\neq\alpha$ in which case $\alpha|_{\a}\in\Sigma$.
		
		It is clear that the subalgebra generated by $\g_{\alpha}$ for $\alpha\in\Sigma$ is $\c(\a)=\m\oplus\a$-stable.  From the decomposition
		\[
		\g=\n_{\Sigma}^-\oplus\c(\a)\oplus\n_{\Sigma}^+,
		\]
		 the result is now clear.
	\end{proof}

	\begin{remark}\label{remark presentation}
		We have shown that for a base $\Sigma=\{\alpha_1,\dots,\alpha_k\}\sub\Delta$, we obtain a presentation of $\g$ such that it is generated by $\c(\a)=\a\oplus\m$, $\g_{\alpha_1},\dots,\g_{\alpha_k},\g_{-\alpha_1},\dots,\g_{-\alpha_k},$ subject to the relations that each $\g_{\alpha_i}$ is a $\c(\a)$-module, and 
		\[
		[\g_{\alpha_i},\g_{-\alpha_j}]=0 \ \text{ for } \ i\neq j.
		\]
		Of course there are more relations, but we don't concern ourselves with this here.  We only mention that $\g$ does not contain any ideals which do not intersect $\h$ because it is Kac-Moody; in particular it does not have any ideals contained in $\n_{\Sigma}^+$.  
	\end{remark}
	
	\begin{cor}
		Let $\alpha\in\Sigma$.  Then $[\g_{\alpha},\g_{\alpha}]=\g_{2\alpha}$.
	\end{cor}
	\begin{proof}
		This follows from Lemma \ref{lemma n gend by}.
	\end{proof}
	
	\begin{cor}\label{cor base sufficient}
		A set $\Sigma\sub\Delta$ is a base if and only if the following conditions hold:
		\begin{enumerate}
			\item $\Sigma$ is linearly independent;
			\item $[\g_{\alpha},\g_{-\beta}]=0$ for distinct $\alpha,\beta\in\Sigma$;
			\item $\g$ is generated by $\c(\a)$ along with the root spaces $\g_{\alpha},\g_{-\alpha}$ where $\alpha$ runs over the elements of $\Sigma$.  
		\end{enumerate}
	\end{cor}
	\begin{proof}
		Indeed, under the conditions we have $\Delta\sub(\N\Sigma\sqcup(-\N\Sigma))$, as required.
	\end{proof}

%
%
%
%

	\subsection{Principal roots}  Observe that the involution $\theta$ defines a symmetric pair $(\g_{\ol{0}},\k_{\ol{0}})$ with the same Cartan subspace $\a$, and along with it a restricted root system $\Delta_0\sub\Delta\sub\a^*$. We call the roots in $\Delta_0$ \emph{even}.  We have that $\Delta_0$ will be a (potentially non-reduced) root system with possibly several irreducible components (Lem.~26.16, \cite{T}).  Suppose that we have a positive system $\Delta_{\Sigma}^+$ for $(\g,\k)$.  Then this induces a positive system $\Delta_{0}=\Delta_{0}^+\sqcup\Delta_{0}^{-}$, and along with it a base $\Pi\sub\Delta_0$.

	\begin{definition}
		Given a base $\Sigma$ of $\Delta$, we call the base $\Pi\sub\Delta_0^+$ determined by $\Sigma$ the \emph{principal roots} (of $\Sigma$).  We note that while $\Pi$ depends on $\Sigma$, we will surpress this dependence in writing for reasons that will become clear.
	\end{definition}
	
	\begin{cor}\label{lemma even part gend}
		The Lie algebra $\g_{\ol{0}}$ is generated by $\m_{\ol{0}}$, $\a$, $(\g_{\alpha})_{\ol{0}},$ and $(\g_{-\alpha})_{\ol{0}}$, where $\alpha$ runs over all principal roots.
	\end{cor}
	
	\begin{proof}
		This follows from Corollary \ref{cor base sufficient} applied to the pair $(\g_{\ol{0}},\k_{\ol{0}})$.
	\end{proof}
	\subsection{Properties of restricted roots}
\begin{definition}
		For a root $\alpha\in\Delta$, write $m_{\alpha}=(m_{\alpha,\ol{0}}|m_{\alpha,\ol{1}})$, where $m_{\alpha,\ol{0}}:=\dim(\g_{\alpha})_{\ol{0}}$, $m_{\alpha,\ol{1}}=\dim(\g_{\alpha})_{\ol{1}}$.
\end{definition}
\begin{lemma}\label{lemma multiples of roots}
	If $\alpha\in\Delta$, then if $k\alpha\in\Delta$ we must have $k\in\{\pm1,\pm2,\pm1/2\}$.  Further, if $2\alpha$ is a root, then $m_{2\alpha,\ol{1}}=0$.
\end{lemma}
\begin{proof}
This follows from the classification presented in Section \ref{section classification}.
\end{proof}

\begin{definition}\label{def real root}
	Following the notation of \cite{RSS}, we say a root $\alpha$ is \emph{regular} if the subalgebra generated by $\g_{\alpha},\g_{-\alpha}$ contains a subalgebra isomorphic to $\s\l(2)$.  We say a root is \emph{singular} if it is not regular.  Note that if a root $\alpha$ is regular then $(\alpha,\alpha)\neq0$, but the converse need not hold.  If $\alpha$ is regular, set $\varepsilon(\alpha)$ to be $0$ if $2\alpha$ is not a root, and otherwise set $\varepsilon(\alpha)=1$.
\end{definition}

We clearly have a decomposition $\Delta=\Delta_{reg}\sqcup\Delta_{sing}$ of roots into regular and singular roots. 

\begin{definition}
	For a nonisotropic root $\alpha\in\Delta$, set $h_{\alpha}:=\frac{2(\alpha,-)}{(\alpha,\alpha)}\in\a$.
\end{definition}

\begin{lemma}\label{lemma root on coroot}
	Let $\alpha$ be a regular root, and suppose that $\beta\in\Delta$.  Then $\beta(h_{\alpha})\in2^{\varepsilon(\alpha)}\Z$.
\end{lemma}

\begin{proof}
	If $\varepsilon(\alpha)=0$, then the statement follows from the representation theory of $\s\l(2)$. If $\varepsilon(\alpha)=1$, we see that $\g_{-2\alpha},\g_{2\alpha}$ will generate a copy of $\s\l(2)$ by Lemma \ref{lemma multiples of roots}, so that \linebreak $\beta(h_{2\alpha})=\frac{1}{2}\beta(h_{\alpha})\in\Z$.  This forces $\beta(h_{\alpha})\in2\Z$.
\end{proof}

	\subsection{Rank one sub-pairs}
	For $\alpha\in\Delta$, set
	\[
	\g\langle\alpha\rangle:=\c(\a)+\bigoplus\limits_{n\in\Q_{\neq0}}\g_{n\alpha}
	\]
	Then $\theta$ stabilizes $\g\langle\alpha\rangle$, and we write $\k\langle\alpha\rangle$ for the fixed subalgebra.  Notice that $(\g\langle\alpha\rangle,\k\langle\alpha\rangle)$ will be a supersymmetric pair of rank 1.

	\subsection{Reflections in simple roots}\label{section reflection simple}
	
	\begin{definition}
			For $\alpha\in\Delta_{reg}$ define $r_{\alpha}\in GL(\a^*)$ to be the usual reflection coming from the baby Weyl group, i.e.~ $r_{\alpha}(\lambda)=\lambda-2\lambda(h_{\alpha})\alpha$. 		
	\end{definition}
	
	Let $\Sigma$ be a base.  For a singular root $\alpha\in\Sigma$, define $r_{\alpha}\Sigma$ to be the base of the positive system $(\Delta^+_{\Sigma}\setminus\{\alpha\})\sqcup\{-\alpha\}$.    Note that the latter is indeed a positive system, as we may define
	\begin{equation}\label{eqn phi sigma alpha}
	\phi_{\Sigma,\alpha}(\alpha)=\epsilon,  \ \ \ \phi_{\Sigma,\alpha}(\beta)=1 \ \text{ for }\beta\in\Sigma\setminus\{\alpha\}
	\end{equation}
	where $\epsilon$ is a small negative number.  Then the positive system obtained from $\phi_{\Sigma,\alpha}$ will be $(\Delta^+_{\Sigma}\setminus\{\alpha\})\sqcup\{-\alpha\}$.
	
	We will call the application of $r_{\alpha}$ to a base a \emph{singular reflection} if $\alpha$ is singular.  
	\begin{lemma}\label{lemma sing refl preserves prin roots}
		Any two bases obtained from one another by a sequence of singular reflections have the same principal roots.
	\end{lemma}
	\begin{proof}
		We see that $(\Delta^+_{\Sigma}\setminus\{\alpha\})\sqcup\{-\alpha\}$ and $\Delta^+$ have the same set of positive even roots, from which the statement follows.
	\end{proof}

	\begin{lemma}\label{lemma sing reflection on base}
		Suppose that $\alpha\in\Sigma$ is simple (singular or regular).  Then $r_{\alpha}\Sigma$ consists of the roots $r_{\alpha}\alpha=-\alpha$ along with $r_{\alpha}\beta$ for $\beta\in\Sigma\setminus\{\alpha\}$, where $r_{\alpha}\beta=\beta+k_{\alpha\beta}\alpha$, where $k_{\alpha\beta}$ is the maximal non-negative integer $k$ such that $\beta+k\alpha\in\Delta$.
	\end{lemma}
	
	\begin{proof}
		Indeed, it is clear that with $\phi_{\Sigma,\alpha}$ as defined in \ref{eqn phi sigma alpha}, the roots $-\alpha$ along with $\beta+k_{\alpha\beta}\alpha$ for $\beta\in\Sigma\setminus\{\alpha\}$ will be both linearly independent and will take the minimum positive values under $\phi_{\Sigma,\alpha}$ as required. 		
	\end{proof}
	
	\begin{remark}
		If $\alpha\in\Sigma$ is singular and $\beta\in\Sigma$, then $r_{\alpha}\beta=\beta+k_{\alpha\beta}\alpha$ for $k_{\alpha\beta}\in\{0,1,2\}$.  This may be checked case by case using the classification of supersymmetric pairs.
	\end{remark}
	
	\begin{lemma}\label{lemma bases equiv}
		Any two bases $\Sigma,\Sigma'$ may be obtained from one another by a sequence of singular reflections and reflections $r_{\alpha}$ for $\alpha\in\Delta_{reg}$.
	\end{lemma}
	\begin{proof}
	If $\Sigma\neq\Sigma'$, then there exists $\alpha\in\Sigma\cap\Delta_{\Sigma'}^-$.  Thus we have $|\Delta^+_{r_{\alpha}\Sigma}\cap\Delta^+_{\Sigma'}|<|\Delta^+_{\Sigma}\cap\Delta^+_{\Sigma'}|$, and we may conclude by induction.
	\end{proof}
	
%
	\subsection{Equivalent positive systems} We partition the collection of bases into equivalence classes, declaring that $\Sigma\sim\Sigma'$ if there exists a sequence of singular reflections $r_{\alpha_1},\dots,r_{\alpha_k}$ such that $\Sigma'=r_{\alpha_k}\cdots r_{\alpha_1}\Sigma$.   Note that a given equivalence class $\mathcal{S}$ of bases has a well-defined set $\Pi$ of principal roots by Lemma \ref{lemma sing refl preserves prin roots}.
	
	\begin{lemma}\label{lemma simple system existence}
		For any base $\Pi\sub\Delta_0$ and any $\gamma\in\Pi$, there exists some base $\Sigma$ for which $\Pi\sub\Delta_{\Sigma}^+$ and either $\gamma\in\Sigma$ or $\gamma/2\in\Sigma$.
	\end{lemma}
	
	\begin{proof}
		We may write $\Z\Delta=Q\oplus P$, $Q,P$ are free $\Z$-modules and $\Z\Pi\sub Q$ is of finite index.  Let $\psi:\Z\Pi\to\R$ be such that $\psi(\gamma)=\epsilon>0$ is very small and positive, and $\psi(\gamma')\gg0$ for $\gamma'\in\Pi\setminus\{\gamma\}$.  Define $\phi:\Z\Delta\to\R$ by extending $\psi$ to $Q$ via the injectivity of $\R$, and letting $\phi$ be very large in absolute value on all non-zero projections of elements of $\Delta$ to $P$.  Then the base obtained from $\phi$ will have the desired properties.
	\end{proof}

	The following is entirely analogous to Cor.~4.5 of \cite{S}, and follows from the same proof as Lemma \ref{lemma bases equiv}
	
	\begin{lemma}\label{lemma principal root equiv}
		If $\Sigma,\Sigma'$ are bases with the same principal roots, then they are equivalent.
	\end{lemma}

	\begin{cor}\label{cor principal root existence}
		Let $\mathcal{S}$ be an equivalence class of bases with principal roots $\Pi$.  Then for every $\gamma\in\Pi$, there exists some base $\Sigma\in\mathcal{S}$ for which either $\gamma$ or $\gamma/2$ is simple.
	\end{cor}
	\begin{proof}
		This follows immediately from Lemmas \ref{lemma simple system existence} and \ref{lemma principal root equiv}.
	\end{proof}	

	\section{Spherical weights}
	
	In this section, we introduce spherical weights and study their behavior under singular reflections.  At this point we chose to abuse the classification of supersymmetric pairs for $\g$ indecomposable and Kac-Moody, as avoiding it seems rather difficult and lacking sufficient payoff.
	
	\subsection{$\Sigma$-spherical weights} Let $\Sigma\sub\Delta$ be a choice of simple roots, and consider the parabolic subalgebra:
	\[
	\q_{\Sigma}=\c(\a)\oplus\n_{\Sigma}=\m\oplus\a\oplus\bigoplus\limits_{\alpha\in\Delta_{\Sigma}^+}\g_{\alpha}.
	\]
	\begin{definition}
		Define $P_{\Sigma}^+\sub\a^*$ to be the those weights $\lambda\in\a^*$ for which there exists a highest weight (with respect to $\q_{\Sigma}$), finite-dimensional $\g$-module $V$ of highest weight $\lambda$, such that $(V^*)^{\k}\neq0$.  We call elements $\lambda\in P_{\Sigma}^+$ the \emph{$\Sigma$-spherical weights}.  Note that $P_{\Sigma}^+$ is a submonoid of $\a^*$ by Cor.~5.10 of \cite{Sh5}.
	\end{definition}	

	\subsection{Parabolic Verma module} For $\lambda\in\a^*$, consider the one dimensional, purely even $\c(\a)$-module $\Bbbk_{\lambda}=\Bbbk\langle v_{\lambda}\rangle$ on which $\a$ acts via $\lambda$ and $\m$ acts by 0.  Inflate $\Bbbk_{\lambda}$ to a module over $\q_{\Sigma}$, and set
	\[
	M_{\Sigma}(\lambda):=\Ind_{\q_{\Sigma}}^{\g}\Bbbk_{\lambda}.
	\]
	The following lemma is standard, and follows from Prop.~5.5.4 and Prop.~5.5.8 of \cite{D}.  
	\begin{lemma}
		There exists a one-dimensional space of $\k$-coinvariants on $M_{\Sigma}(\lambda)$.  In particular, there exists a minimal quotient of $M_{\Sigma}(\lambda)$ which continues to admit a nonzero $\k$-coinvariant.
	\end{lemma}
	
	\begin{definition}
		We write $V_{\Sigma}(\lambda)$ for the minimal quotient of $M_{\Sigma}(\lambda)$ which continues to admit a nonzero $\k$-coinvariant.
	\end{definition}

\noindent\textbf{Caution:} $V_{\Sigma}(\lambda)$ need not be irreducible! 	
	
	\begin{lemma}\label{lemma iwasawa parabolic verma}
		We have $\UU\k\cdot v_{\lambda}=M_{\Sigma}(\lambda)$, so that:
		\[
		M_{\Sigma}(\lambda)=\Bbbk\langle v_{\lambda}\rangle\oplus \k\UU\k\cdot v_{\lambda}.
		\]
		In particular, if $\varphi:M_{\Sigma}(\lambda)\to\Bbbk$ is a nontrivial $\k$-coinvariant, we have $\varphi(v_{\lambda})\neq0$, and $\varphi(v)=0$ if and only if $v\in \k\UU\k v_{\lambda}$.
	\end{lemma}
	\begin{proof}
		This follows from the Iwasawa decomposition $\g=\k\oplus\a\oplus\n_{\Sigma}$, and the fact that $\a\oplus\n_{\Sigma}\sub\q_{\Sigma}$.
	\end{proof}
	
	\begin{lemma}\label{lemma restriction Verma simple root}
		Let $\lambda\in\a^*$ and $\alpha\in\Sigma$.  Then the $\g\langle\alpha\rangle$-submodule of $M_{\Sigma}(\lambda)$ generated by $v_{\lambda}$ is isomorphic to $M_{\{\alpha\}}(\lambda)$. Further a non-zero $\k$-coinvariant on $M_{\Sigma}(\lambda)$ restricts to a non-zero $\k\langle\alpha\rangle$-coinvariant on $M_{\{\alpha\}}(\lambda)$.
	\end{lemma}
	
	\begin{proof}
		The first claim is by the PBW theorem.  For the second claim, write $\varphi$ for a non-zero coinvariant on $M_{\Sigma}(\lambda)$.  Lemma \ref{lemma iwasawa parabolic verma} tells us that $\varphi(v_{\lambda})\neq0$.  Since $v_{\lambda}\in \UU\g\langle\alpha\rangle v_{\lambda}$, it is clear that $\varphi$ restricts to a nonzero coinvariant on this submodule.
	\end{proof}

	\begin{remark}
	Let $\mathcal{G}$ be a quasireductive supergroup which is a global form of $\g$ and is such that $\theta$ lifts to an involution of $\mathcal{G}$.  Let $\KK\sub\mathcal{G}$ be a subgroup satisfying $(\mathcal{G}^\theta)^\circ\sub\KK\sub\mathcal{G}^\theta$, where $(\mathcal{G}^\theta)^\circ$ denotes the connected component of the identity of $\mathcal{G}^\theta$.  Notice that $\KK$ will also be quasireductive.	Then we call the coset space $\mathcal{G}/\KK$ a supersymmetric space (see \cite{MT} for the construction of homogeneous superspaces). 
	
	Given an Iwasawa Borel subalgebra $\b$, that is, a Borel subalgebra containing $\a\oplus\n$, we may consider the $\b$-eigenfunctions in $\Bbbk[\mathcal{G}/\KK]$.  This set will exactly be those $\lambda\in P_{\Sigma}^+$ for which $V_{\Sigma}(\lambda)$ integrates to a representation of the group $\mathcal{G}$ and for which $\KK_0$ acts trivially on the $\k$-coinvariant of $V_{\Sigma}(\lambda)$.  Further, for exactly such $\lambda$ we will have an embedding $V_{\Sigma}(\lambda)\sub\Bbbk[\mathcal{G}/\KK]$.
	\end{remark}


%
%
%
\subsection{Integrability}

	\begin{definition}
		For $\alpha\in\Delta$, we say that a $\g$-module $V$ is $\alpha$-integrable if both $\g_{\alpha}$ and $\g_{-\alpha}$ act locally nilpotently on $V$, i.e.~if for each $v\in V$ there exists $N>0$ such that $(\g_{\alpha})^Nv=0$ and $(\g_{-\alpha})^Nv=0$.
	\end{definition}

	\begin{lemma}\label{lemma integrability}
		For $\lambda\in\a^*$, write $v_{\lambda}'$ for a highest weight vector of $V_{\Sigma}(\lambda)$.  Then the following are equivalent:
		\begin{enumerate}
			\item $\lambda\in P_{\Sigma}^+$;
			\item $V_{\Sigma}(\lambda)$ is finite-dimensional;
			\item for every regular root $\alpha\in\Delta$, there exists an $N>0$ such that $(\g_{\alpha})^Nv_{\lambda}'=0$;
			\item for every $\alpha\in\Pi$, there exists an $N>0$ such that $(\g_{-\alpha})^Nv_{\lambda}'=0$.
		\end{enumerate}     
	\end{lemma}
	
	\begin{proof}
			Here we use that $\UU\g$ is a $\alpha$-integrable for all $\alpha\in\Delta$ under the adjoint action.  For the last equivalence, we use Corollary \ref{lemma even part gend}.
	\end{proof}

\subsection{Detecting singular subspaces}\label{remark singular subspaces} 

\begin{lemma}\label{lemma detecting sing subspaces}
	Suppose that $\lambda\in\a^*$ and $\alpha\in\Sigma$, and write $\varphi$ for a nonzero $\k$-coinvariant on $M_{\Sigma}(\lambda)$.  If  $W\sub\UU\g\langle\alpha\rangle\cdot v_{\lambda}$ is a $\g\langle\alpha\rangle$ submodule of $M_{\Sigma}(\lambda)$ such that $\varphi(W)=0$, then the natural map $M_{\Sigma}(\lambda)\to V_{\Sigma}(\lambda)$ factors over $\UU\g\cdot W$.
\end{lemma}

\begin{proof}
	Recall from Remark \ref{remark presentation} that for two distinct simple roots $\alpha,\beta\in\Sigma$, we have
	\[
	[\g_{-\alpha},\g_{\beta}]=0.
	\]
	Therefore since $W$ is $\g_{\alpha}$-stable it will also be $\n_{\Sigma}$-stable, meaning that $\q_{\Sigma}W\sub W$.  By the Iwasawa decomposition we deduce that $\UU\g\cdot W=\UU\k\cdot W$, which implies that $\varphi$ vanishes on the $\g$-submodule generated by $W$, implying the result.  
\end{proof}

\subsection{Classification of rank one supersymmetric pairs}\label{section classification}
	For the following classification result, we refer to the classification of supersymmetric pairs, which is described in Sec.~5.2 of \cite{Sh}, and is based off \cite{S3}.

	\begin{lemma}\label{lemma classification rank one}
	Suppose that $(\g,\k)$ is rank one, so that $\Sigma=\{\alpha\}$.  Then up to split factors fixed by $\theta$, we may write $\g=\g'\times \a'$, where $\a'$ denotes a complimentary abelian subalgebra on which $\theta$ acts by $(-1)$, and $(\g',\k)$ is one of the following:  
	\begin{enumerate}
		\item $\alpha$ is singular:
		\begin{enumerate}
			\item[(i)]  $(\alpha,\alpha)\neq0$, $m_{\alpha}=(0|2n)$ for $n\geq 1$, and $m_{2\alpha}=(0|0)$
			\[
			(\o\s\p(2|2n),\o\s\p(1|2n))
			\]
			\item[(ii)] $(\alpha,\alpha)=0$, $m_{\alpha}=(0|2)$, and $m_{2\alpha}=(0|0)$
			\[
			(\g\l(1|1)\times\g\l(1|1),\g\l(1|1))	
			\]
		\end{enumerate}
	
		\item $\alpha$ is regular: 
		\begin{enumerate}
		\item[(iii)] $m_{\alpha}=(m-2|2n)$ for $m\geq 3$, $n\geq 1$, and $m_{2\alpha}=(0|0)$
		\[
		(\o\s\p(m|2n),\o\s\p(m-1|2n))
		\]
		\item[(iv)] $m_{\alpha}=(4(n-2)|2m)$ for $m\geq 1$, $n\geq 2$, and $m_{2\alpha}=(3|0)$, for $n\geq 2$
		\[
		(\o\s\p(m|2n),\o\s\p(m|2n-2)\times\s\p(2))
		\]
		\item[(v)] $m_{\alpha}=(2(m-2)|2n)$ for $m\geq 2$, $n\geq 1$, and $m_{2\alpha}=(1|0)$
		\[
		(\g\l(m|n),\g\l(m-1|n)\times\g\l(1))
		\]
		\item[(vi)] $m_{\alpha}=(2|0)$ and $m_{2\alpha}=(0|0)$
		\[
		(\s\l(2)\times\s\l(2),\s\l(2))
		\]
		\item[(vii)] $m_{\alpha}=(0|2)$ and $m_{2\alpha}=(2|0)$
		\[
		(\o\s\p(1|2)\times\o\s\p(1|2),\o\s\p(1|2))
		\]
		\item[(viii)]  $m_{\alpha}=(8|0)$, and $m_{2\alpha}=(7|0)$
		\[
		(\mathfrak{f}(4),\s\o(9))
		\]
	\end{enumerate}
\end{enumerate}
\end{lemma}

\begin{remark}
	The purely even rank one symmetric pair $(\mathfrak{f}(4),\s\o(9))$\footnote{Thank you to an anonymous referee who pointed out this gap.} will not be relevant for our purposes as it does not appear as a rank one pair $(\g\langle\alpha\rangle,\k\langle\alpha\rangle)$ for any pair $(\g,\k)$ when $\g$ is an indecomposable Kac-Moody superalgebra with $\g_{\ol{1}}\neq0$.  However, since even symmetric pairs are well understood, we include it for completeness.  
\end{remark}

\begin{cor}
	Let $(\g,\k)$ be a supersymmetric pair and let $\alpha\in\Delta$. Then $(\g\langle\alpha\rangle,\k\langle\alpha\rangle)$ is isomorphic to one of the pairs in the list of Lemma \ref{lemma classification rank one} after quotienting by split, $\theta$-fixed ideals.
\end{cor}

	\subsection{Rank one integrability conditions} 		
	
	\begin{lemma}\label{lemma rank one red}
		If $\Sigma$ is a base and $\alpha\in\Sigma$ is a regular root, then $V_{\Sigma}(\lambda)$ is $\alpha$-integrable if and only if $\lambda\in P_{\{\alpha\}}^+$ with respect to the rank one root system determined by $(\g\langle\alpha\rangle,\k\langle\alpha\rangle)$.  
	\end{lemma}
	\begin{proof}
		We continue to write $v_{\lambda}$ for the highest weight vector of $M_{\Sigma}(\lambda)$.  By Lemma \ref{lemma restriction Verma simple root}, we have a natural inclusion of $\g\langle\alpha\rangle$-modules $M_{\{\alpha\}}(\lambda)\cong\UU\g\langle\alpha\rangle v_{\lambda}\sub M_{\Sigma}(\lambda)$, and a nonzero $\k$-coinvariant on $M_{\Sigma}(\lambda)$ restricts to a nonzero $\k\langle\alpha\rangle$-coinvariant on $M_{\{\alpha\}}(\lambda)$.  
		
		Write $N_{\alpha}\sub M_{\{\alpha\}}(\lambda)$ for the kernel of the quotient map $M_{\{\alpha\}}(\lambda)\to V_{\{\alpha\}}(\lambda)$, and similarly write $N_{\Sigma}$ for the kernel of $M_{\Sigma}(\lambda)\to V_{\Sigma}(\lambda)$.  	We claim that $N_{\alpha}=N_{\Sigma}\cap M_{\{\alpha\}}(\lambda)$.  Indeed, $N_{\alpha}(\lambda)$ satisfies the hypotheses of Lemma \ref{lemma detecting sing subspaces}, which implies that $N_{\alpha}\sub N_{\Sigma}$.  The reverse inclusion is clear from the definition.
		
		It follows that we have an inclusion $V_{\{\alpha\}}(\lambda)\sub V_{\Sigma}(\lambda)$.  Since these modules share the same highest weight vector, one is $\alpha$-integrable if and only if the other is, completing the proof.
	\end{proof}
	
	\begin{thm}\label{lemma integrability rank one}
		Let $\alpha\in\Sigma$ be a regular, simple root.  Then a weight $\lambda\in\a^*$ is $\alpha$-integrable if and only if $\lambda(h_{\alpha})\in 2^{\varepsilon(\alpha)}\cdot 2\Z_{\geq0}$
	\end{thm}
	\begin{proof}
		By Lemma \ref{lemma rank one red}, we may assume that $(\g,\k)$ is rank one so that $\Sigma=\{\alpha\}$, and we now make this assumption.
		
	Suppose that $\lambda$ is $\alpha$-integrable, so that $V_{\Sigma}(\lambda)$ is a finite-dimensional $\g$-module.  Then $\lambda$ will be a spherical weight for the underlying even symmetric pair of rank one, which by Thm.~3.12 of \cite{L} implies that $\lambda(h_{\alpha})\in 2^{\varepsilon(\alpha)}\cdot2\Z_{\geq0}$.  
	
	For the converse, we use the classification given in Section \ref{section classification}.  For every $\lambda\in\a^*$ satisfying $\lambda(h_{\alpha})=2^{\varepsilon(\alpha)}\cdot 2$, we construct an explicit  spherical, highest weight representation $V$ with highest weight $\lambda$.  This is enough because $P_{\Sigma}^+$ is a monoid.
	\begin{enumerate}
		\item For $(\o\s\p(m|2n),\o\s\p(m-1|2n))$, the standard representation $\Bbbk^{m|2n}$ is irreducible and spherical because it admits a trivial summand upon restriction to $\k$.   
		\item For $(\o\s\p(m|2n),\o\s\p(m|2n-2)\times\s\p(2))$ the irreducible representation $(S^2\Bbbk^{m|2n})/\Bbbk$ is spherical because the restriction of $S^2\Bbbk^{m|2n}$ to $\k$ admits a two-dimensional space of coinvariants corresponding to the two invariant bilinear forms.  After quotienting by $\Bbbk$ we will still have one $\k$-coininvariant, giving sphericity.
		\item For $(\g\l(m|n),\g\l(m-1|n))$ we consider the adjoint representation $V=\s\l(m|n)$.  This admits a nontrivial $\k$-coinvariant given by the trace of the $(m-1+n)\times (m-1+n)$-submatrix corresponding to the inclusion $\g\l(m-1|n)\sub\g\l(m|n)$.  This representation is a highest weight module for all $(m,n)$.  
		\item The diagonal cases $(\g\times\g,\g)$ are easy.
		\item The case $(\mathfrak{f}(4),\s\o(9))$ follows from the classical Cartan-Helgason theorem.
	\end{enumerate}
	\end{proof}

	\subsection{Spherical weights for $\Delta=\Delta_{reg}$ or  $\Pi\sub\Sigma$}  In the case $\Delta=\Delta_{reg}$, we have that for all $\alpha\in\Pi$, either $\alpha\in\Sigma$ or $\alpha/2\in\Sigma$.  Thus in this case we obtain:
	
	\begin{thm}\label{theorem all roots real}
		Suppose that $\Delta_{reg}=\Delta$.  Then for a base $\Sigma$, we have $\lambda\in\a^*$ is $\Sigma$-spherical if and only if for all $\alpha\in\Sigma$ we have $\lambda(h_{\alpha})\in 2^{\varepsilon(\alpha)}\cdot 2\Z$.
	\end{thm}
	
	Theorem \ref{theorem all roots real} applies to the following pairs:
	\[
	(\g\l(m|n),\g\l(r)\times\g\l(m-r|n)), \ \ r\leq m/2, \ \ (\o\s\p(m|2n),\o\s\p(r|2n)\times\s\o(r)), \ \ r<m/2,
	\]
	\[
	(\o\s\p(m|2n),\o\s\p(m|2n-2s)\times\s\p(2s)), \ \ s\leq n/2, \ \ (\a\g(1|2),\mathfrak{d}(2,1;3)).
	\]
	The following is also clear.
	\begin{prop}\label{prop type I}
		Suppose that $\Sigma$ is a set of simple roots such that $\Pi\sub\Sigma$. Then a weight $\lambda\in\a^*$ is $\Sigma$-spherical if and only if for all $\alpha\in\Pi$ we have $\lambda(h_{\alpha})\in2^{\varepsilon(\alpha)}\cdot 2\Z_{\geq0}$.
	\end{prop}

	Proposition \ref{prop type I} applies to the following pairs for particular choices of base $\Sigma$:
	\[
	(\g\l(m|2n),\o\s\p(m|2n)), \ \ \  (\o\s\p(2|2n),\o\s\p(1|2r)\times\o\s\p(1|2n-2r)).
	\]
	\subsection{Singular reflections of highest weights}

	\begin{lemma}\label{lemma iso refl}
		Let $\alpha\in\Sigma$ be a singular root with $(\alpha,\alpha)=0$, and let $\lambda\in\a^*$.  Then \linebreak $V_{\Sigma}(\lambda)\cong V_{r_{\alpha}\Sigma}(r_{\alpha}\lambda)$, where $r_{\alpha}\lambda\in\a^*$ is given by the following formula:
		\begin{equation*}
			r_{\alpha}\lambda=
			\begin{cases}
				\lambda-2\alpha & \text{ if }(\lambda,\alpha)\neq0 \\
				\lambda & \text{ if }(\lambda,\alpha)=0.			
			\end{cases}
		\end{equation*}
	\end{lemma}

	\begin{proof}
		Recall by Lemma \ref{lemma classification rank one} that we have $\g\langle\alpha\rangle\cong\g\l(1|1)\times\g\l(1|1)\times\a'$ up to split factors fixed by $\theta$, and the involution swaps the two factors of $\g\l(1|1)$, and is $(-1)$ on $\a'$,.  Thus we may write $\lambda=(\lambda_0,-\lambda_0,\lambda')$, where $\lambda_0$ is a weight of $\g\l(1|1)$ and $\lambda'$ is a weight of $\a'$.  We may similarly write $\alpha=(\alpha',-\alpha',0)/2$, where $\alpha'$ is a root of $\g\l(1|1)$ with coroot $h_{\alpha'}$.  Now we see that
		\[
		(\lambda,\alpha)=\lambda_0(h_{\alpha'}).
		\]
		
	First suppose that $(\lambda,\alpha)=0$, which is equivalent to $\lambda_0(h_{\alpha'})=0$.  Let $e_1,e_2$ be the raising operators and $f_1,f_2$ the lowering operators of the two copies of $\g\l(1|1)$.  Thus $\g_{-\alpha}=\Bbbk\langle f_1,e_2\rangle$.  Then if we set $W:=\Bbbk\langle f_1v_{\lambda},e_2v_{\lambda}\rangle\sub M_{\Sigma}(\lambda)$, we see that it is $\m$-stable and $\g_{\alpha}W=0$.   Further, $W$ is purely odd so that the $\k$-coinvariant vanishes on it.  Thus Lemma \ref{lemma detecting sing subspaces} implies that the map $M_{\Sigma}(\lambda)\to V_{\Sigma}(\lambda)$ factors over the quotient by $\UU\g\cdot W$.  In particular, $\g\langle\alpha\rangle$ stabilizes the image of $v_{\lambda}$ in $V_{\Sigma}(\lambda)$.  From this we see that $V_{\Sigma}(\lambda)$ is a quotient of $M_{r_{\alpha}\Sigma}(\lambda)$, and so by universality we have $V_{\Sigma}(\lambda)\cong V_{r_{\alpha}\Sigma}(\lambda)$.

	On the other hand, $(\lambda,\alpha)\neq0$ is equivalent to $\lambda_0(h_{\alpha}')\neq0$.  In this case we see that $f_1e_2v_{\lambda}\in M_{\Sigma}(\lambda)$ will be of weight $\lambda-2\alpha$ with respect to $\a$, and will be a $\p_{r_{\alpha}\Sigma}$-singular vector.  Thus we obtain a map $M_{r_{\alpha}\Sigma}(\lambda-2\alpha)\to M_{\alpha}(\lambda)$. Surjectivity is easy to check, and so we get an isomorphism $M_{r_{\alpha}\Sigma}(\lambda-2\alpha)\cong M_{\alpha}(\lambda)$.  From this we easily obtain $V_{\Sigma}(\lambda)\cong V_{r_{\alpha}\Sigma}(\lambda-2\alpha)$.
	\end{proof}
	
	We now state what happens in our `best-behaved' case.  Note the following proposition is effectively a generalization of Thm.~10.5 of \cite{S}.
	
	\begin{prop}\label{proposition all singular isotropic}
		Suppose that every singular root $\alpha\in\Delta_{sing}$ is isotropic.  Let $\Pi\sub\Delta_{\Sigma}^+$ be the set of principal roots and let $\lambda\in\a^*$.  Then $\lambda\in P_{\Sigma}^+$ if and only if for each $\gamma\in\Pi$ there exists some base $\Sigma'\sim\Sigma$ such that $\gamma\in\Sigma'$ and the corresponding reflected weight $\lambda_{\Sigma'}$ satisfies $\lambda_{\Sigma'}(h_{\gamma})\in 2^{\varepsilon(\gamma)}\cdot 2\Z_{\geq0}$.
	\end{prop}
	
	\begin{proof}
		This follows from Corollary \ref{cor principal root existence} and \ref{lemma iso refl}.
	\end{proof}
	
 	Proposition \ref{proposition all singular isotropic} applies to the following pairs: 
	\[
	(\g\times\g,\g), \ \ \ \ (\g\l(m|n),\g\l(r|s)\times\g\l(m-r|n-s))
	\]

	\subsection{Reflections in nonisotropic, singular roots}
\begin{lemma}\label{lemma noniso noncrit reflection}
	Suppose that $\alpha\in\Sigma$ is a singular, nonisotropic root of multiplicity $(0|2n)$, and $\lambda\in\a^*$.  
	\begin{enumerate}
		\item If $\lambda(h_{\alpha})/2\notin\{0,1,\dots,n-1,n+1,\dots,2n\}$, then $V_{\Sigma}(\lambda)\cong V_{r_{\alpha}\Sigma}(\lambda-2n\alpha)$;
		\item if $k:=\lambda(h_{\alpha})/2\in\{0,1,\dots,n-1\}$ then 
		\[
		V_{\Sigma}(\lambda)\cong V_{r_{\alpha}\Sigma}(\lambda-2k\alpha);
		\]
		\item if $\lambda(h_{\alpha})/2=n+k\in\{n+1,n+2,\dots,2n\}$, then $V_{\Sigma}(\lambda)$ contains $V_{\Sigma}(\lambda-2k\alpha)$ as a submodule.  In particular, if $\lambda\in P^+_{\Sigma}$ then $\lambda-2k\alpha\in P_{\Sigma}^+$.
	\end{enumerate}
\end{lemma}

\begin{proof}
	Observe that $\g_{\alpha}$ will be an irreducible, purely odd, $\m$-module with a symplectic form $\omega\in\Lambda^2\g_{-\alpha}$.  Thus $\omega^n\in\Lambda^{2n}\g_{-\alpha}=\Lambda^{top}\g_{-\alpha}$ is non-zero.
	
	In case (1), $\lambda$ will be a typical weight for $\o\s\p(2|2n)\sub\g\langle\alpha\rangle$.  Thus $\omega^nv_{\lambda}$ will generate $M_{\Sigma}(\lambda)$, is of weight $\lambda-2n\alpha$, and is annihilated by $\m$ and $\g_{-\alpha}$.  Hence we obtain a map $M_{r_{\alpha}\Sigma}(\lambda-2n\alpha)\to M_{\Sigma}(\lambda)$, and it is easy to see it is an isomorphism.  In this way, we obtain an isomorphism $V_{\Sigma}(\lambda)\cong V_{r_{\alpha}\Sigma}(\lambda-2n\alpha)$, as desired.
	
	In case (2), we use the computations in Sec.~10 of \cite{Sh2} to see that $\UU\g\langle\alpha\rangle v_{\lambda}\sub M_{\Sigma}(\lambda)$ contains a $\q_{\Sigma}$-stable subspace $W$ on which the $\k$-coinvariant vanishes. Write $M'$ for the the quotient of $M_{\Sigma}(\lambda)$ by the submodule generated by $W$, and write $v_{\lambda}'$ for the image of $v_{\lambda}$ in $M'$.  Then by construction, $\UU\g\langle\alpha\rangle v_{\lambda}'$ is an irreducible $\UU\g\langle\alpha\rangle$-submodule of $M'$ with lowest weight vector $\omega^kv_{\lambda}$ (again applying computations in \cite{Sh2}).  Since $\omega^kv_{\lambda}$ will be a highest weight vector with respect to $r_{\alpha}\Sigma$, we obtain a surjective map $M_{r_{\alpha}\Sigma}(\lambda-2k\alpha)\to M'$, from which it easily follows that $V_{\Sigma}(\lambda)\cong V_{r_{\alpha}\Sigma}(\lambda-2k\alpha)$.
	
	Finally, for case (3), the computations of Sec.~10 in \cite{Sh2} once again show that $\omega^kv_{\lambda}$ is an $\m$-fixed, $\p$-singular vector on which the $\k$-coinvariant does not vanish. Therefore we get a map $V_{\Sigma}(\lambda-k\alpha)\to V_{\Sigma}(\lambda)$, and injectivity is by definition of $V_{\Sigma}(\lambda-k\alpha)$.
\end{proof}

\begin{definition}
	Let $\lambda\in\a^*$ and let $\alpha\in\Sigma$ be a singular, nonisotropic root with $m_{\alpha}=(0|2n)$.  Then we say that $\lambda$ is an $\alpha$-critical weight if $\lambda(h_{\alpha})/2\in\{n+1,\dots,2n\}$.
\end{definition}
\begin{definition}\label{def sing refl non iso}
	Let $\lambda\in\a^*$, and let $\alpha\in\Sigma$ be a simple, singular, non-isotropic root.  If $\lambda$ is not $\alpha$-critical, then set:
	\[
	r_{\alpha}\lambda:=\begin{cases}
		\lambda-\lambda(h_{\alpha})\alpha & \text{ if }\lambda(h_{\alpha})/2\in\{0,\dots,n-1\};\\
		\lambda-2n\alpha & \text{ otherwise}.
	\end{cases}
	\]
	In particular, by Lemma \ref{lemma noniso noncrit reflection}, $V_{\Sigma}(\lambda)\cong V_{r_{\alpha}\Sigma}(r_{\alpha}\lambda)$ for such $\lambda$.
\end{definition}

Using singular reflections, it is clear that we can understand when a weight $\lambda\in\a^*$ lies in $P_{\Sigma}^+$ if no reflection of $\lambda$ is critical for simple, singular, non-isotropic roots.  For weights that are critical for some singular, nonisotropic $\alpha\in\Sigma$, we may use the following to gain some traction.

\begin{lemma}\label{lemma refl noniso crit}
	Let $\alpha\in\Sigma$ be a nonisotropic, singular root, and let $\beta$ be a simple regular root of $r_{\alpha}\Sigma$. Then for an $\alpha$-critical weight $\lambda\in\a^*$,  $V_{\Sigma}(\lambda)$ is $\beta$-integrable if and only if both $V_{r_{\alpha}\Sigma}(\lambda)$ and $V_{r_{\alpha}\Sigma}(\lambda-2n\alpha)$ are.
\end{lemma}
\textbf{Caution}: we do \emph{not} have an isomorphism $V_{\Sigma}(\lambda)\cong V_{r_{\alpha}\Sigma}(\lambda)$.

\begin{proof}	
	Let $\p_{\alpha,\Sigma}$ be the parabolic subalgebra of $\g$ containing both $\q_{\Sigma}$ and $\g_{-\alpha}$.  Let $V_{\lambda,\alpha}$ be the finite-dimensional Kac-module over $\g\langle\alpha\rangle\cong\o\s\p(2|2n)\times\a'\times(...)$ of highest weight $\lambda$, and write $v_{\lambda,\alpha}$ for its highest weight vector. Note that $V_{\lambda,\alpha}$ is indecomposable with composition series $0\to L_{\lambda-2k,\alpha}\to V_{\lambda,\alpha}\to L_{\lambda,\alpha}\to0$, where $L_{(-),\alpha}$ is the corresponding simple module over $\g\langle\alpha\rangle$.  
		
		Then we may inflate $V_{\lambda,\alpha}$ to $\p_{\alpha,\Sigma}$.  Observe that $V_{\Sigma}(\lambda)$ is a quotient of $\Ind_{\p_{\alpha,\Sigma}}^{\g}V_{\lambda,\alpha}$.  Write $v_{\lambda,\alpha}'$ for the image of $v_{\lambda,\alpha}$ in $V_{\Sigma}(\lambda)$.  As a module over $\p_{r_{\alpha}\Sigma}$, $V_{\lambda,\alpha}$ is generated by $V_{\lambda,\alpha}^{\m}=\Bbbk\langle v_{\lambda,\alpha},\omega v_{\lambda,\alpha},\cdots,\omega^{n}v_{\lambda,\alpha}\rangle$.   Thus $V_{\Sigma}(\lambda)$ is $\beta$-integrable if and only if each vector $\omega^jv_{\lambda,\alpha}'$ is $\beta$-integrable, i.e.~$(\g_{-\beta})^N\omega^jv_{\lambda,\alpha}'=0$ for $N\gg0$.  However, by Theorem \ref{lemma integrability rank one}, this is in turn is equivalent to:
		\[
		(\lambda-2j\alpha)(h_{\beta})\in2^{\varepsilon(\beta)}\cdot2\Z_{\geq0} \ \text{ for all }j=0,\dots,n.
		\]
		Since evaluation at $h_{\beta}$ is a linear function in $j$, the result is now clear from Lemma \ref{lemma root on coroot}.
\end{proof}

We note that in principal one could use the idea of Lemma \ref{lemma refl noniso crit} to try and understand what happens after performing multiple reflections.  However, we don't see at this moment how to understand this picture in a simple way in any generality.

\section{Explicit computations of spherical weights}

In the final section, we explicitly describe the sets $P_{\Sigma}^+$ for convenient choices of $\Sigma$ when $(\g,\k)$ is a supersymmetric pair with $\g$ indecomposable and $\g_{\ol{1}}\neq0$.  In the problematic cases, i.e.~those for which nonisotropic singular roots are present, we choose $\Sigma$ so that all but the last necessary singular reflection will have critical weights, allowing us to rely on Lemma \ref{lemma refl noniso crit}.

We stress that we have made choices of seemingly convenient bases $\Sigma$. However the techniques we employ have the potential to work for other bases as well. If one is interested in a description of $P_{\Sigma}^+$ for a $\Sigma$ not used here, one may attempt to apply our techniques in their case.  In particular, our technique will work when $\Sigma$ has the property that every principal root either lies in $\Sigma$ or lies in $r_{\alpha}\Sigma$ for some singular root $\alpha\in\Sigma$.  Such $\Sigma$ can often be constructed by making them consists of as many singular roots as possible.

We begin by working abstractly with restricted root systems, for simplicity.

\begin{example}\label{example BC}
	Let $r,s\in\Z_{\geq0}$, $k\in\Bbbk^\times$, and define $BC_{k}(r,s)$ to be the weak generalized root system with two regular components $\Delta_{reg}=BC_r\sqcup BC_s$ and singular roots $\Delta_{sing}=W(\omega_1^{(1)}+\omega_1^{(2)})$.  Here $\omega_i^{(j)}$ denotes the $i$th fundamental weight of the $j$th irreducible component of $\Delta_{reg}$, as in \cite{S2}.  The bilinear form is normalized such that $(\omega_1^{(1)},\omega_1^{(1)})=1$ and $(\omega_{1}^{(2)},\omega_{1}^{(2)})=k$.  Define $C_k(r,s)\sub BC_k(r,s)$ to be the same as $BC_k(r,s)$ but without the short roots.  
	
	Let us write $\gamma_1,\dots,\gamma_r,\nu_1,\dots,\nu_s$ for a basis of the underlying vector space given by mutually orthogonal short roots, where $\gamma_1,\dots,\gamma_r\in BC_r$ and $\nu_1,\dots,\nu_s\in BC_s$.  Then one base $\Sigma$ for $BC_{k}(r,s)$ is given by:
	\[
	\gamma_1-\gamma_2,\dots,\gamma_{r-1}-\gamma_r,\gamma_r-\nu_1,\dots,\nu_{s-1}-\nu_s,\nu_s.
	\]
	A set of principal roots is given by $\Pi=\{	\gamma_1-\gamma_2,\dots,\gamma_{r-1}-\gamma_r,\gamma_r,\nu_1-\nu_2,\dots,\nu_{s-1}-\nu_s,\nu_s\}$.  Therefore we have that $\{\gamma_r\}=\Pi\setminus(\Pi\cap\Sigma)$. 
	
	For $C_k(r,s)$, we may take the same base and principal roots only we need to multiply $\nu_s$ and $\gamma_r$ by $2$.  However this minor difference does not affect any of the computations of $P_{\Sigma}^+$ below.
	
	Let $\lambda=\sum\limits_{i}a_i\gamma_i+\sum\limits_jb_j\nu_j$.  Then a necessary condition for $\lambda\in P_{\Sigma}^+$ is that $\lambda(h_{\alpha})\in 2^{\varepsilon(\alpha)}\cdot 2\Z_{\geq0}$ for all $\alpha\in\Pi$, i.e.
	\begin{equation}\label{eqn numeric assumptions}
	a_i-a_{i+1},a_r,b_{i}-b_{i+1},b_s\in2\Z_{\geq0}.
	\end{equation}
	In particular we assume (\ref{eqn numeric assumptions}).  To obtain $\gamma_r$ as a simple root, we must apply the singular reflections $r_{\gamma_r-\nu_1},\dots,r_{\gamma_r-\nu_s}$.  We see that for $k\neq-1$:
	\[
	\lambda(h_{\gamma_r-\nu_1})/2=\frac{a_r-kb_1}{1+k}.
	\]
		
	\textbf{Case I, $k=-1$:}  In this case all singular roots are isotropic, so in fact Proposition \ref{proposition all singular isotropic} applies.  We see that $(\lambda,\gamma_r-\nu_1)=a_r+b_1$, which is non-negative, and zero if and only if $a_r=b_1=0$.  Thus either $b_i=0$ for all $i$ in which case applying odd reflections show that $\lambda\in P_{\Sigma}^+$, or $b_1\neq0$ so that $r_{\gamma_r-\nu_1}\lambda=\lambda-2\gamma_r+2\nu_1$, meaning we must have $a_r\geq 2$.  If $b_2=0$ then again it is clear that $\lambda\in P_{\Sigma}^+$, and otherwise we need $a_{r}\geq4$. Continuing like this, we learn that 
\[
\lambda\in P_{\Sigma}^+\iff a_r/2\geq|\{i:b_i\neq0\}|.
\]
	
	\textbf{Case II, $k=-1/2$, $m_{\gamma_r-\nu_1}=(0|2)$:}  We have $\gamma_i-\nu_j$ is a non-isotropic singular root.  Observe that $\lambda(h_{\gamma_{r}-\nu_1})/2=2a_r+b_1$ which is an even, non-negative integer.  If this quantity is 0, then $b_i=0$ for all $i$ and $a_r=0$, so that $\lambda\in P_{\Sigma}^+$.  
	
	If $a_r=0$ and $b_1>0$ then by Lemma \ref{lemma noniso noncrit reflection}, $\lambda\in P_{\Sigma}^+$ only if $\lambda-2\gamma_r+2\nu_1$ is integrable.  But it is clearly not since the coefficient of $\gamma_r$ would become negative. Thus $a_r=0\Rightarrow b_1=0$, and for all other values of $(a_r,b_1)$ we have $2a_r+b_1>2$. Hence there are no critical weights for $\gamma_r-\nu_1$ since $m_{\gamma_r-\nu_1}=(0|2)$.  In the case when $a_r,b_1\neq0$, we have $r_{\gamma_r-\nu_1}\lambda=\lambda-2\gamma_r+2\nu_1$, meaning we must have $a_r\geq2$.  
		
	Using inductive reasoning as in the first case, we once again find that:
	\[
	\lambda\in P_{\Sigma}^+\iff a_r/2\geq|\{i:b_i\neq0\}|.
	\]
	
	\textbf{Case III, $r=s=1$, $m_{\gamma_1-\nu_1}=(0|2)$, $k\neq -1$:} 

		Recall that 
		\[
		\lambda(h_{\gamma_1-\nu_1})/2=\frac{a_1-kb_1}{1+k}.
		\]
		If the above quantity is equal to 0, then we must have $a_1=kb_1$, implying that either $a_1=b_1=0$ or $a_1,b_1>0$.  If $\lambda(h_{\gamma_1-\nu_1})=2$, then we must have $a_1\neq 2$.  However if $a_1=0$ then $\lambda-2\gamma_1+2\nu_1$ is not dominant, so $\lambda$ can't be either by Lemma \ref{lemma noniso noncrit reflection}.  
		
		In all other cases, $r_{\alpha}\lambda=\lambda-2\gamma_1+2\nu_1$ must be integrable, meaning that $a_1\geq2$.  Therefore we obtain that:
		\[
		\lambda=a_1\gamma_1+b_1\nu_1\in P_{\Sigma}^+\iff a_1=0\Rightarrow b_1=0.
		\]
	
\end{example}

\subsection{Tables with $P_{\Sigma}^+$}\label{section tables} In Table \ref{table 1}, we describe $\Delta$ and make a choice of base $\Sigma$ for each pair. In Table \ref{table 2} which follows, we explicitly describe $P_{\Sigma}^+$ for the given choice of $\Sigma$.  We use the presentations of generalized root systems given in Sec.~5.2 of \cite{Sh}.  In the following section we will justify our computations.

\renewcommand{\arraystretch}{1.3}
\begin {table}

\begin{center}
\caption {}\label{table 1} 
\begin{tabular}{|c|c|}
	\hline
	$(\g,\k)$ & $\Delta$\\
	& $\Sigma$\\
	\hline
	$(\g\l(m|2n),\o\s\p(m|2n))$  & $A(m-1,n-1)$ \\    
	& $\epsilon_1-\epsilon_2,\dots,\epsilon_m-\nu_1,\nu_1-\nu_2,\dots,\nu_{n-1}-\nu_n$\\
	& $\nu_i:=(\delta_{2i-1}+\delta_{2i})/2$\\
	\hline
	$(\g\l(m|n),\g\l(r|s)\times\g\l(m-r|n-s))$ & $(B)C_{-1}(r,s)$\\ 
	$r\leq m/2, s\leq n/2$ & $\gamma_1-\gamma_2,\dots,\gamma_{r}-\nu_1,\nu_1-\nu_2,\dots,\nu_{s-1}-\nu_s,(2)\nu_s$\\
	& $\gamma_i:=(\epsilon_i-\epsilon_{m-i+1})/2$, $\nu_i:=(\delta_i-\delta_{n-i+1})/2$\\
	\hline 
	$(\o\s\p(2m|2n),\g\l(m|n))$ & $(B)C_{-1/2}(n,m)$\\
	& $\delta_1-\delta_2,\dots,\delta_{n}-\gamma_1,\gamma_1-\gamma_2,\dots,\gamma_{m-1}-\gamma_m,(2)\gamma_m$\\
	& $\gamma_i=(\epsilon_{2i-1}+\epsilon_{2i})/2$\\
	\hline
	$(\o\s\p(m|2n),\o\s\p(r|2s)\times\o\s\p(m-r|2n-2s))$ & $BC_{-1/2}(r,s)$ \\
	$r<m/2, s\leq n/2$                                   & $\epsilon_1-\epsilon_2,\dots,\epsilon_r-\nu_1,\nu_1-\nu_2,\dots,\nu_{s-1}-\nu_s,\nu_s$\\
	    & $\nu_i=(\delta_{2i-1}+\delta_{2i})/2$\\
	\hline
	$(\o\s\p(2r|2n),\o\s\p(r|2s)\times\o\s\p(r|2n-2s))$ & $\Delta_{reg}=D_{r}\sqcup BC_s$\\
	          $s<n/2$                                   & $\Delta_{sing}=W\omega_{1}^{(1)}$ if $s=0$, \\
	                                                    & $\Delta_{sing}=W(\omega_{1}^{(1)}+\omega_{1}^{(2)})\sqcup W\omega_{1}^{(1)}$ if $s>0$\\
	                                                    & $\nu_1-\nu_2,\dots,\nu_{s}-\epsilon_1,\dots,\epsilon_{r-1}-\epsilon_r,\epsilon_r$\\
	& $\nu_i=(\delta_{2i-1}+\delta_{2i})/2$\\
	\hline 
	$(\o\s\p(2r|4s),\o\s\p(r|2s)\times\o\s\p(r|2s))$ & $D(r,s)$\\
	& $\nu_1-\nu_2,\dots,\nu_s-\epsilon_1,\epsilon_1-\epsilon_2,\dots,\epsilon_{r-1}-\epsilon_r,\epsilon_{r-1}+\epsilon_r$ \\
	& 	$\nu_i=(\delta_{2i-1}+\delta_{2i})/2$\\
	\hline
	$(\mathfrak{d}(2,1;a),\o\s\p(2|2)\times\s\o(2))$    & $C_{a}(1,1)$, $a\neq -1$ \\
	& $\alpha-\beta, 2\beta$\\
	\hline
	$(\a\b(1|3),\s\l(1|4))$                         & $C_{-3}(1,1)$ \\
	& $\epsilon/2-\delta/2,\delta$\\
	\hline
	$(\a\b(1|3),\g\o\s\p(2|4))$                             & $\Delta_{reg}=B_2\sqcup C_1$, $\Delta_{sing}=W(\omega_{2}^{(1)}+\omega_1^{(2)})$\\
	& $\epsilon_2,(\epsilon_1-\epsilon_2-\delta)/2,\delta$\\
	\hline
	$(\a\b(1|3),\mathfrak{d}(2,1;2)\times\s\l(2))$                   & $\Delta_{reg}=B_3$, $\Delta_{sing}=W\omega_3$\\
	& $\epsilon_2-\epsilon_3, \epsilon_1-\epsilon_2, (-\epsilon_1+\epsilon_2+\epsilon_3)/2$\\
	\hline
	$(\a\g(1|2),\mathfrak{d}(2,1;3))$                   & $G_2$ \\
	& any base\\
	\hline
\end{tabular}
\end{center}
\end {table}

\begin{table}
	\caption {}\label{table 2} 
	\begin{center}

\begin{tabular}{|c|c|}
	\hline
	$(\g,\k)$ & $P_{\Sigma}^+$\\
	\hline
	$(\g\l(m|2n),\o\s\p(m|2n))$ & $a_1\epsilon_1+\dots+a_m\epsilon_m+b_1\nu_1+\dots+b_n\nu_n$ \\ 
	                            & $a_i-a_{i+1},b_{i}-b_{i+1}\in2\Z_{\geq0}$\\
	\hline
	$(\g\l(m|n),$ & $a_1\gamma_1+\dots+a_r\gamma_r+b_1\nu_1+\dots+b_s\gamma_s$\\
	$\g\l(r|s)\times\g\l(m-r|n-s))$ & $a_i-a_{i+1},a_r,b_i-b_{i+1},b_s\in 2\Z_{\geq0}, \  a_r/2\geq|\{i:b_i\neq0\}|$ \\
	$r\leq m/2$, $s\leq n/2$ & \\
	\hline 
	$(\o\s\p(2m|2n),\g\l(m|n))$ & $a_1\gamma_1+\dots+a_m\gamma_m+b_1\delta_1+\dots+b_n\delta_n$\\
	                            & $a_i-a_{i+1},a_m,b_i-b_{i+1},b_n\in 2\Z_{\geq0}, \  b_n/2\geq|\{i:a_i\neq0\}|$\\
	\hline
	$(\o\s\p(m|2n),$ & $a_1\epsilon_1+\dots+a_r\epsilon_r+b_1\nu_1+\dots+b_s\nu_s$ \\
	$\o\s\p(r|2s)\times\o\s\p(m-r|2n-2s))$             & $a_i-a_{i+1},a_r,b_i-b_{i+1},b_s\in 2\Z_{\geq0}, \  a_r/2\geq|\{i:b_i\neq0\}|$ \\
	$r<m/2, s\leq n/2$  &\\
	\hline
	$(\o\s\p(2r|2n),$                                   & $a_1\epsilon_1+\dots+a_r\epsilon_r+b_1\nu_1+\dots+b_s\nu_s$ \\
	$\o\s\p(r|2s)\times\o\s\p(r|2n-2s))$                                          & $a_i-a_{i+1},a_{r-1}+a_{r},b_i-b_{i+1},b_s\in 2\Z_{\geq0}$. \\
                                         $s<n/2$   	&  $b_{s}/2\geq|\{i<r:a_i\neq 0\}|$ and  \\
                                         				& $a_{r-1}+a_r\geq 2n-4s$ if $a_r<0$\\
	\hline
	$(\o\s\p(2r|4s),$ & $a_1\epsilon_1+\dots+a_r\epsilon_r+b_1\nu_1+\dots+b_s\nu_s$ \\
	            $\o\s\p(r|2s)\times\o\s\p(r|2s))$       & $a_i-a_{i+1},a_{r-1}+a_r,b_i-b_{i+1},b_s\in 2\Z_{\geq0}$, \\
	                                                    & $b_s/2\geq |\{i:a_i\neq0\}|$\\
	\hline
$(\mathfrak{d}(2,1;a),\o\s\p(2|2)\times\s\o(2))$    & $a\alpha+b\beta$ \\
& $a,b\in 2\Z_{\geq0}$, $a=0\Rightarrow b=0$\\
\hline
$(\a\b(1|3),\s\l(1|4))$                         & $a\epsilon+b\delta$ \\
& $a,b\in\Z_{\geq0}$ and either $a=b=0$ or $a\geq 2$\\
\hline
$(\a\b(1|3),\g\o\s\p(2|4))$                             & $a_1\epsilon_1+a_2\epsilon_2+b\delta$\\
& $a_1-a_2\in2\Z_{\geq0}$, $a_2,b\in\Z_{\geq0}$\\
& either $a_1=a_2=b=0$ or $a_1>a_2$\\
\hline
$(\a\b(1|3),\mathfrak{d}(2,1;2)\times\s\l(2))$                   & $a_1\epsilon_1+a_2\epsilon_2+a_3\epsilon_3$ \\
& $a_i-a_{i+1}\in2\Z_{\geq0}$, $a_3\in\Z_{\geq0}$\\
& and either $a_3=0$ or $a_1>a_2$ \\
\hline
$(\a\g(1|2),\mathfrak{d}(2,1;3))$                   & $a_1\omega_1+a_2\omega_2$ \\
& $a_1,a_2\in 2\Z_{\geq0}$,\\
&  $\omega_1,\omega_2$ the fundamental dominant weights.\\
\hline
\end{tabular}
	\end{center}
\end{table}

\newpage

\

\

\

\subsection{Computations of $P_{\Sigma}^+$ for supersymmetric pairs} We now justify the computations presented in the tables above by going through the pairs case by case.  We once again remind that we use the presentations of generalized root systems given in Sec.~5.2 of \cite{Sh}.

\begin{enumerate}
	\item $(\g\l(m|2n),\o\s\p(m|2n))$: For the choice of $\Sigma$ given we have $\Pi\sub\Sigma$, so we may apply Proposition \ref{prop type I} to obtain our description.
	
	\item $(\g\l(m|n),\g\l(r|s)\times\g\l(m-r|n-s))$, $r\leq m/2,$ $s\leq n/2$: In this case we obtain the restricted root system $BC_{-1}(r,s)$ if $r<n/2$ or $s<m/2$, and otherwise we get $C_{-1}(m/2,n/2)$.  Thus we may apply Case I of Example \ref{example BC} to obtain our description.

	\item $(\o\s\p(2m|2n),\g\l(m|n))$:  After renormalizing the form, we obtain $C_{-1/2}(n,m)$ if $m$ is even and $BC_{-1/2}(n,m)$ otherwise, and in both cases the multiplicities of singular roots are $2$.  Thus we may apply Case II of Example \ref{example BC} to obtain our description.
	
	\item $(\o\s\p(m|2n),\o\s\p(m-r|2n-2s)\times\o\s\p(r|2s))$, $r<m/2$, $s\leq n/2$: We obtain the restricted root system $BC_{-1/2}(r,s)$, with singular root multiplicities $2$, so we may apply Case II of Example \ref{example BC}.
	
	\item $(\o\s\p(2r|2n),\o\s\p(r|2s)\times\o\s\p(r|2n-2s))$, $s<n/2$: let us be more explicit about the restricted root system.  We have $(\epsilon_i,\epsilon_j)=\delta_{ij}$, $(\epsilon_i,\nu_j)=0$, and $(\nu_i,\nu_j)=-\delta_{ij}/2$.  The singular roots $W\epsilon_1$ have multiplicity $2n-4s$, and for $s>0$ the singular roots $W(\epsilon_1+\delta_1)$ have multiplicity $2$.  We have chosen the base:
	\[
	\nu_1-\nu_2,\dots,\nu_{s-1}-\nu_s,\nu_s-\epsilon_1,\epsilon_1-\epsilon_2,\dots,\epsilon_{r-1}-\epsilon_r,\epsilon_r.
	\]
	The principal roots are given by:
	\[
	\epsilon_1-\epsilon_2,\dots,\epsilon_{r-1}\pm\epsilon_{r},\nu_1-\nu_2,\dots,\nu_{s-1}-\nu_s,\nu_s.
	\]	
	Thus we need to apply the singular reflection $\epsilon_r$ to obtain $\epsilon_{r-1}+\epsilon_r$, and separately we need to apply the sequence of singular reflections $r_{\nu_s-\epsilon_1},r_{\nu_s-\epsilon_2},\dots,r_{\nu_s-\epsilon_{r-1}},r_{\epsilon_r}$ to obtain $\nu_s$.
	
	Let $\lambda=\sum\limits_{i=1}^{r}a_i\epsilon_i+\sum\limits_{j=1}^{s}b_j\nu_j$.  Then by integrability with respect to the even symmetric pair, we must have $a_i-a_{i+1},a_{r-1}+a_{r},b_i-b_{i+1},b_{s}\in 2\Z_{\geq0}$.  We see that $\lambda(h_{\epsilon_r})/2=a_r$, from which we deduce that if $a_r<0$ we must have $a_{r-1}\geq-(a_r-(2n-4s))$, or equivalently $a_{r-1}+a_{r}\geq2n-4s$.
	
	On the other hand, applying the singular reflections $r_{\nu_s-\epsilon_1},r_{\nu_s-\epsilon_2},\dots,r_{\nu_s-\epsilon_{r-1}}$, we see that we never encounter critical weights, and we obtain the condition $b_s/2\geq |\{i<r:a_i\neq0\}|$.  After applying these singular reflections, the coefficient of $\epsilon_{r-1}$ is either 0 or $a_{r-1}+2$.  Thus our condition that $a_{r-1}+a_r\geq2n-4s$ if $a_r<0$ will still hold for these new coefficients, meaning this is sufficient.
	
	\item $(\o\s\p(2r|4s),\o\s\p(r|2s)\times\o\s\p(r|2s))$: we take the base $\nu_1-\nu_2,\dots,\nu_{s-1}-\nu_s,\nu_s-\epsilon_1,\dots,\epsilon_{r-1}\pm\epsilon_r$, and the principal roots 
	\[
	\epsilon_1-\epsilon_2,\dots,\epsilon_{r-1}\pm\epsilon_{r}, \nu_1-\nu_2,\dots,\nu_{s-1}-\nu_s,2\nu_s.
	\] 
	The singular roots $W(\epsilon_1-\delta_1)$ have multiplicity 2.  Then we need to apply the simple reflections $r_{\nu_s-\epsilon_1},r_{\nu_s-\epsilon_2},\dots,r_{\nu_s-\epsilon_r}$ to obtain $2\nu_s$.  Let $\lambda=\sum\limits_jb_j\nu_j+\sum\limits_ia_i\epsilon_i$ with the necessary conditions $a_i-a_{i+1},a_{r-1}+a_{r},b_i-b_{i+1},b_{s}\in 2\Z_{\geq0}$.  No weights are critical when applying the first $r-1$ singular reflections.  If $b_s=2k$ with $k<r-1$, we are forced to have $a_{k+1}=\dots=a_r=0$, and we will have $\lambda\in P_{\Sigma}^+$.  
	
	Suppose instead that $b_s=2k$ with $k\geq r-1$.  Then 
	\begin{eqnarray*}
	\lambda'& := &r_{\nu_s-\epsilon_{r-1}}\dots r_{\nu_s-\epsilon_1}\lambda \\        
	        & =  &b_1\nu_1+\dots+b_{s-1}\nu_{s-1}+(b_s-2(r-1))\nu_s+(a_1+2)\epsilon_1+\dots+(a_{r-1}+2)\epsilon_{r-1}+a_r\epsilon_r.
	\end{eqnarray*}
	We see that
	\[
	\lambda'(h_{\nu_s-\epsilon_r})=-(b_s-2(r-1))-2a_r.
	\]
	If this quantity is $0$, then $\lambda\in P_{\Sigma}^+$.  Otherwise, the following weight must be integrable:
	\[
	b_1\nu_1+\dots+b_{s-1}\nu_{s-1}+(b_s-2r)\nu_s+(a_1+2)\epsilon_1+\dots+(a_{r-1}+2)\epsilon_{r-1}+(a_r+2)\epsilon_r,
	\]
	meaning we need $b_s\geq 2r$.

	\item $(\mathfrak{d}(2,1;a),\o\s\p(2|2)\times\s\o(2))$: Here we obtain restricted root system $C_a(1,1)$ with $a\neq -1$, and where nonisotropic singular roots have multiplicity $2$.  Thus we may apply Case III of Example \ref{example BC}.
	
	\item $(\mathfrak{ab}(1|3),\s\l(1|4))$: this case is equivalent, after globally rescaling the form, to $C_{-3}(1,1)$ deformed: $\a^*$ has basis $\epsilon,\delta$ where $(\epsilon,\epsilon)=1/3$, $(\delta,\epsilon)=0$, $(\delta,\delta)=-1$.  Singular roots have multiplicity $4$ here. Then $\Sigma$ can be taken as $(\epsilon-\delta)/2,\delta$ and $\Pi=\{\epsilon,\delta\}$.  We have $r_{(\epsilon-\delta)/2}\Sigma=\{(\delta-\epsilon)/2,\epsilon\}$.  Any nonzero $\lambda=a\epsilon+b\delta$ with $a,b\geq0$ will have $\lambda(h_{(\epsilon-\delta)/2})<0$, so we do not have any critical weights.  Thus $\lambda\in P_{\Sigma}^+$ if and only if $a,b\in\Z$ and either $a=b=0$ or $a\geq 2$.
	
	\item $(\mathfrak{ab}(1|3),\g\o\s\p(2|4))$: $\Delta_{reg}=B_2\sqcup C_1$ with deformed bilinear form: $\a^*$ has basis $\epsilon_1,\epsilon_2,\delta$ where $(\epsilon_i,\epsilon_j)=\delta_{ij}/3$, $(\epsilon_i,\delta)=0$, and $(\delta,\delta)=-1$.  Singular roots have multiplicity $2$. We take $\Sigma$ to be $\epsilon_2,(\epsilon_1-\epsilon_2-\delta)/2,\delta$, and we have $\Pi=\{\epsilon_1-\epsilon_2,\epsilon_2,\delta\}$.  Then $r_{(\epsilon_1-\epsilon_2-\delta)/2}\Sigma$ is $(\epsilon_1+\epsilon_2-\delta)/2,(-\epsilon_1+\epsilon_2+\delta)/2,\epsilon_1-\epsilon_2$.  
	
	Now if $\lambda=a_1\epsilon_1+a_2\epsilon_2+b\delta$, then $\lambda(h_{(\epsilon_1-\epsilon_2-\delta)/2})<0$ whenever $a_1,a_2,b\geq0$ and any one of them is positive.  So we avoid any critical weights again.  From this we compute that $\lambda\in P_{\Sigma}^+$ if and only if $a_2,b\in\Z_{\geq0}$, $a_1-a_2\in2\Z_{\geq0}$, and either $a_{1}=a_{2}=b=0$ or $a_{1}>a_2$.
	
	\item $(\mathfrak{ab}(1|3),\mathfrak{d}(2,1;2)\times\s\l(2))$: $B_3$ with small orbit: $\a^*$ has basis $\epsilon_1,\epsilon_2,\epsilon_3$, where $(\epsilon_i,\epsilon_j)=\delta_{ij}/3$.  Singular roots have multiplicity $2$ in this case. For $\Sigma$ we take $\epsilon_2-\epsilon_3,\epsilon_1-\epsilon_2,(-\epsilon_1+\epsilon_2+\epsilon_3)/2$, and $\Pi=\{\epsilon_1-\epsilon_2,\epsilon_2-\epsilon_3,\epsilon_3\}$.  Then we have $r_{(-\epsilon_1+\epsilon_2+\epsilon_3)/2}\Sigma$ is given by $\epsilon_2-\epsilon_3,\epsilon_3,(\epsilon_1-\epsilon_2-\epsilon_3)/2$.
	
	Let $\lambda=a_1\epsilon_1+a_2\epsilon_2+a_3\epsilon_3\in\a^*$.  Then a necessary condition that $\lambda\in P_{\Sigma}^+$ is that $a_1-a_2,a_2-a_3\in2\Z_{\geq0}$, $a_3\in\Z_{\geq0}$.  We see that
	\[
	\lambda(h_{(-\epsilon_1+\epsilon_2+\epsilon_3)/2})/2=\frac{2}{3}(-a_1+a_2+a_3).
	\] 
	Thus any weight $\lambda=(a+b)\epsilon_1+a\epsilon_2+b\epsilon_3$ such that $a,b,a-b\in 2\Z_{\geq0}$ will lie in $P_{\Sigma}^+$.  Otherwise, we need that $a_1>a_2$.  This is equivalent to the condition given in the table.
	
	\item $(\mathfrak{ag}(1|2),\mathfrak{d}(2,1;3))$: $G_2$: In this case the restricted root system is just $G_2$, and all simple roots are regular. We present $\a^*$ with basis $\nu_1,\nu_2$, the fundamental weights for $G_2$.  Then for $\lambda=a_1\nu_1+a_2\nu_2\in\a^*$, we have $\lambda\in P_{\Sigma}^+$ if and only if $a_1-a_2,a_2\in 2\Z_{\geq0}$.

\end{enumerate}

	\bibliographystyle{amsalpha}

\begin{thebibliography}{99999999}
	
	\bibitem{AS} A.~Alldridge and S.~Schmittner, \emph{Spherical representations of Lie supergroups}. J.~Funct.~Anal., 268.6 (2015): 1403-1453.
	
	
	
	
	\bibitem{C} K.~Coulembier, \emph{The Orthosymplectic Superalgebra in Harmonic Analysis}.  J.~Lie Theory, Vol.~23 (2013): 55–83.
	
	\bibitem{D} J.~Dixmier, \emph{Enveloping algebras}. No.~11, Amer.~Math.~Soc.~(1996).
	
	\bibitem{H1} S.~Helgason, \emph{A duality for symmetric spaces with applications to group representations}. Adv.~Math., 5.1 (1970): 1-154.
	
	\bibitem{H2} S.~Helgason, \emph{Differential geometry and symmetric spaces}. Amer.~Math.~Soc., Vol.~341 (2001).
	
	\bibitem{LS} R.Lávička and D.~Šmíd, \emph{Fischer decomposition for polynomials on superspace}. J.~Math.~Phys., 56.11 (2015).
	
	\bibitem{L} J.~Lepowsky, \emph{Generalized Verma modules, the Cartan-Helgason theorem, and the Harish-Chandra homomorphism}. J.~Algebra, 49.2 (1977): 470-495.
	
	\bibitem{MT} A. Masuoka and Y. Takahashi, \emph{Geometric construction of quotients $G/H$ in supersymmetry}.  Transform.~Groups, 26.1 (2021): 347–375.
	
	\bibitem{RSS} Reif, Shifra, Siddhartha Sahi, and Vera Serganova. \emph{Restriction Theorems and Root Systems for Symmetric Superspaces}.  Indag.~Math.~(2024), https://doi.org/10.1016/j.indag.2024.09.006.
	
	\bibitem{SSS} S.~Sahi, H.~Salmasian, and V.~Serganova, \emph{The Capelli eigenvalue problem for Lie superalgebras}. Math.~Z., 294.1-2 (2020): 359-395.
	
	\bibitem{SSS2} S.~Sahi, H.~Salmasian, and V.~Serganova, \emph{Capelli operators for spherical superharmonics and the Dougall–Ramanujan identity}. Transform.~Groups, 27.4 (2022): 1475-1514.
	
	\bibitem{SaSo} S.~Sahi, and S.~Zhu, \emph{Supersymmetric Shimura operators and interpolation polynomials}. Preprint arXiv:2312.08661 (2023).
	
	\bibitem{SRFL}  A.~Schnyder, S.~Ryu, A.~Furusaki and A.~Ludwig, \emph{Classification of topological
	insulators and superconductors}. AIP Conf.~Proc., 1134 (2009): 10–21.
	
	\bibitem{SeV} A.~Sergeev and A.~Veselov, \emph{Deformed quantum Calogero-Moser problems and Lie superalgebras}. Comm.~Math.~Phys., 245 (2004): 249-278.
	
	\bibitem{S} V.~Serganova, \emph{Kac–Moody superalgebras and integrability}. Developments and trends in infinite-dimensional Lie theory (2011): 169-218.
	
	\bibitem{S2} V.~Serganova, \emph{On generalizations of root systems}. Comm.~Algebra, 24.13 (1996): 4281-4299.
	
	\bibitem{S3} V.~Serganova, \emph{Automorphisms of simple Lie superalgebras}.  Mathematics of the USSR-Izvestiya, 24.3 (1985): 539.
	
	\bibitem{Sh} A.~Sherman, \emph{Iwasawa decomposition for Lie superalgebras}. J.~Lie Theory, Vol. 32 (2022): 973-996.
	
	\bibitem{Sh2} A.~Sherman, \emph{Spherical indecomposable representations of Lie superalgebras}. J.~Algebra,  Vol.~547 (2020): 262-311.
	
	\bibitem{Sh3} A.~Sherman, \emph{Spherical and Symmetric supervarieties}. PhD Thesis, UC Berkeley.
	
	\bibitem{Sh5} A.~Sherman, \emph{Spherical supervarieties}. Ann. Inst. Fourier (Grenoble), Vol.~71, No.~4 (2021).
	
	\bibitem{ShSi} A.~Sherman and L.~Silberberg, \emph{A queer Kac-Moody construction}.  Preprint arXiv:2309.09559.
	
	\bibitem{T} D.~Timashev, \emph{Homogeneous spaces and equivariant embeddings}. Springer Science \& Business Media, Vol.~138 (2011).
	
	\bibitem{Z} M.~R.~Zirnbauer, \emph{Riemannian symmetric superspaces and their origin in random matrix theory}, J.~Math.~Phys., 37, no. 10 (1996): 4986–5018.
\end{thebibliography}

\textsc{\footnotesize Alexander Sherman, School of Mathematics and Statistics, University of Sydney, Camperdown NSW 2006} 

\textit{\footnotesize Email address:} \texttt{\footnotesize xandersherm@gmail.com}

\end{document}